\documentclass[preprint,11pt]{elsarticle}

\usepackage{mathtools}
\usepackage{amssymb}
\usepackage{amsthm}
\usepackage[section]{placeins}
\usepackage{tikz}
\usepackage{tikz,pgfplots}
\usepackage{multirow}
\usetikzlibrary{decorations.markings}

\newcommand{\trace}{Tr}
\newcommand{\aut}{Aut}
\newcommand{\fix}{Fix}
\newcommand{\FIX}{FIX}
\newcommand{\lcm}{lcm}

\usepackage[margin=2.5cm]{geometry}

\journal{Discrete Mathematics}

\begin{document}

	\newdefinition{definition}{Definition}
	\newtheorem{theorem}{Theorem}
	\newtheorem{corollary}{Corollary}
	\newtheorem{lemma}{Lemma}
	\newtheorem{conjecture}{Conjecture}
	\newdefinition{observation}{Observation}
	\newdefinition{problem}{Problem}
	\tikzset{middlearrow/.style={
			decoration={markings,
				mark= at position 0.7 with {\arrow[scale=2]{#1}} ,
			},
			postaction={decorate}
		}
	}
	
	\tikzset{midarrow/.style={
			decoration={markings,
				mark= at position 0.5 with {\arrow[scale=2]{#1}} ,
			},
			postaction={decorate}
		}
	}
	
	\newcommand{\TODO}[1]{\textcolor{red}{TODO: #1}}

	\begin{frontmatter}
		
		
		\title{The structure of digraphs with excess one}
		


		\author[label1]{James Tuite}
		\ead{james.tuite@open.ac.uk}

		\address{Department of Mathematics and Statistics, Open University, Walton Hall, Milton Keynes}

		\begin{abstract}
		A digraph $G$ is \emph{$k$-geodetic} if for any (not necessarily distinct) vertices $u,v$ there is at most one directed walk from $u$ to $v$ with length not exceeding $k$. The order of a $k$-geodetic digraph with minimum out-degree $d$ is bounded below by the directed Moore bound $M(d,k) = 1+d+d^2+\dots +d^k$. The Moore bound can be met only in the trivial cases $d = 1$ and $k = 1$, so it is of interest to look for $k$-geodetic digraphs with out-degree $d$ and smallest possible order $M(d,k)+\epsilon $, where $\epsilon $ is the \emph{excess} of the digraph. Miller, Miret and Sillasen recently ruled out the existence of digraphs with excess one for $k = 3,4$ and $d \geq 2$ and for $k = 2$ and $d \geq 8$. We conjecture that there are no digraphs with excess one for $d,k \geq 2$ and in this paper we investigate the structure of minimal counterexamples to this conjecture. We severely constrain the possible structures of the outlier function and prove the non-existence of certain digraphs with degree three and excess one, as well closing the open cases $k = 2$ and $d = 3,4,5,6,7$ left by the analysis of Miller et al. We further show that there are no involutary digraphs with excess one, i.e. the outlier function of any such digraph must contain a cycle of length $\geq 3$.     	
		\end{abstract}
		
		
	\end{frontmatter}

	\section{Introduction}
The undirected degree/girth problem asks for the smallest possible order of a graph that has degree $d$ and girth $g$. If $g = 2k+1$ is odd, then the order of any such graph is bounded below by the undirected Moore bound $1+d+d(d-1)+\dots +d(d-1)^{k-1}$. A graph that has order exceeding the Moore bound by $\epsilon $ has \emph{excess} $\epsilon $; a graph that meets the Moore bound is called a \emph{Moore graph}. It was shown in \cite{BanIto1,Dam,HofSin} that, with the trivial exception of cycles and complete graphs, for odd $g = 2k+1$ the Moore bound can be met only for diameter $k = 2$ and degrees $d = 3$ (the Petersen graph), $d = 7$ (the Hoffman-Singleton graph) and possibly $d = 57$. Interestingly, it is even harder for a graph to have excess one than to meet the Moore bound; Bannai and Ito showed in \cite{BanIto} that there are no graphs with excess one and girth $2k+1 \geq 5$.

A directed graph $G$ is \emph{$k$-geodetic} if for any pair of (not necessarily distinct) vertices $u,v$ of $G$ there is at most one directed walk from $u$ to $v$ in $G$ with length $\leq k$. The following analogue for directed graphs of the degree/girth problem was raised by Sillasen in \cite{Sil}:

 \begin{problem}[Degree/geodecity problem]
 	What is the smallest possible order of a $k$-geodetic digraph with minimum out-degree $d$?
 	\end{problem}
It was shown in \cite{TuiErs} that for any $d,k \geq 1$ there exists a diregular $k$-geodetic digraph with degree $d$, so this problem is well-posed.

A lower bound on the order of a $k$-geodetic digraph $G$ with minimum out-degree $d$ can be derived as follows. Fix a vertex $u$ of $G$; this will be Level 0 of a directed \emph{Moore tree} of depth $k$. There are at least $d$ out-neighbours of $u$; draw these at Level 1 of the tree with arcs from $u$ at Level 0. In general once we have completed Level $t$ of the tree, where $0 \leq t \leq k-1$, we draw arcs from each vertex of Level $t$ to each of its out-neighbours at Level $t+1$. It follows by induction that for $0 \leq t \leq k$ there are at least $d^t$ vertices at Level $t$ of the Moore tree. Moreover by $k$-geodecity all of the vertices in this tree must be distinct. It follows that the order of $G$ is bounded below by the \emph{directed Moore bound} \[ M(d,k) = 1+ d+ d^2+\dots +d^k.\] It was shown by Bridges and Toueg in \cite{BriTou} that the only digraphs that meet the directed Moore bound are directed cycles and complete digraphs. Hence for $d,k \geq 2$ we are interested in $k$-geodetic digraphs with minimum out-degree $d$ and order $M(d,k)+\epsilon $ for some small \emph{excess} $\epsilon $. A smallest such digraph is a \emph{geodetic cage}.

In this paper we will be interested in digraphs with excess $\epsilon = 1$; we will call such a digraph a $(d,k;+1)$-digraph, which we will typically denote by the letter $G$. For every vertex $u$ of such a digraph there is a unique vertex $o(u)$ such that the distance from $u$ to $o(u)$ in $G$ satisfies $d(u,o(u)) \geq k+1$; we call the associated function $o:V(G) \rightarrow V(G)$ the \emph{outlier function} of $G$. The following results on the structure of $(d,k;+1)$-digraphs were proven in \cite{MilMirSil,Sil}.   

\begin{lemma}\cite{MilMirSil,Sil}\label{lem:known results}
	If $G$ is a $(d,k;+1)$-digraph, then
	\begin{itemize}
		\item $G$ is diregular.
		\item Either $G$ is a $(d,2;+1)$-digraph, where $d$ lies in the range $3 \leq d \leq 7$, or a $(d,k;+1)$-digraph with $d \geq 3$ and $k \geq 5$, or a directed $(k+2)$-cycle. 
		\item The outlier function $o$ of $G$ is an automorphism.  
	\end{itemize} 
\end{lemma}
Further results on the problem and examples of geodetic cages can be found in \cite{Tui,Tui2,Tui3,TuiErs}. As no non-trivial digraphs with excess one have been found, we make the following conjecture.

\begin{conjecture}\label{no digraphs with excess one}
	There are no $(d,k;+1)$-digraphs with $d,k \geq 2$.
\end{conjecture}

It is instructive to compare this problem with the state of the art of the directed degree/diameter problem (a survey of which can be found in \cite{MilSir}), which asks for the largest digraphs with given maximum degree $d$, diameter $k$ and order $M(d,k)-\delta $ for some small \emph{defect} $\delta $. A digraph with defect $\delta = 1$ is a \emph{$(d,k;-1)$-digraph} or an \emph{almost Moore digraph}. 

In the Moore tree of a $(d,k;-1)$-digraph rooted at a vertex $u$ there will be a unique vertex $r(u)$ that is repeated in the tree; this vertex is called the \emph{repeat} of $u$ and the associated \emph{repeat function} $r:V(G) \rightarrow V(G)$ is also a digraph automorphism \cite{BasMilPleZna}. All $(d,k;-1)$-digraphs are diregular \cite{MilGimSirSla}.  By contrast with the degree/geodecity problem, for all $d \geq 3$ there is a unique $(d,2;-1)$-digraph \cite{FioYebAle,Gim,Gim2}; however there are no $(d,k;-1)$-digraphs for $d \geq 2$ and diameters $k = 3,4$ \cite{ConGimGonMirMor,ConGimGonMirMor2}. The existence of $(2,k;-1)$-digraphs with $k \geq 3$ was ruled out in \cite{MilFri}, as were $(3,k;-1)$-digraphs with $k \geq 3$ in \cite{BasMilSirSut}.

The plan of this paper is as follows. Firstly in Section \ref{section:VT} we use counting arguments to deduce some strong conditions on digraphs with excess one and a high level of symmetry. In Section \ref{section:degree three} we investigate the structure of digraphs with degree three and excess one. Then in Section \ref{Automorphisms} we use the approach of Sillasen in \cite{Sil2} to analyse the set of vertices fixed by an automorphism of a $(d,k;+1)$-digraph. As the outlier function $o$ is an automorphism of $G$ these results tell us quite a lot about the structure of $o$ as a permutation. In the final part of this paper, Section \ref{section:spectrum digraphs excess one}, we use a spectral approach to exploit the results of Section \ref{Automorphisms} to prove the non-existence of certain $(d,k;+1)$-digraphs.

\section{Vertex-transitive digraphs with excess one}\label{section:VT}

All of the known Moore graphs are vertex-transitive. This suggests that it is of interest to look for digraphs with order close to the directed Moore bound that have a high degree of symmetry. In her thesis \cite{SilThesis} Sillasen uses this approach on digraphs with defect one; we now emulate this approach for digraphs with excess one.  

In \cite{SilThesis} as a basis for her counting arguments Sillasen divides the vertices of an almost Moore digraph into two types, Type I and Type II. Adapting this notation, we make the following definition.  

\begin{definition}
	A vertex $u$ of a $(d,k;+1)$-digraph is Type II if $o(u) \rightarrow u$; otherwise $u$ is Type I. 	
\end{definition}
The type of a vertex is preserved by any automorphism of $G$. This leads us to the following observation.

\begin{observation}
	If a $(d,k;+1)$-digraph $G$ is vertex-transitive, then either every vertex of $G$ is Type I or every vertex of $G$ is Type II. 
\end{observation}

If every vertex of $G$ is Type I, then we can obtain a strong divisibility condition on $d$ and $k$. The reason for this is that all directed $(k+1)$-cycles of $G$ are arc-disjoint. 

\begin{lemma}\label{no more than one k-cycle}
	No arc of $G$ is contained in more than one directed $(k+1)$-cycle.
\end{lemma}
\begin{proof}
	Suppose that an arc $(u,v)$ is contained in two $(k+1)$-cycles.  Then there are two distinct $k$-paths from $v$ to $u$, which contradicts $k$-geodecity.
\end{proof}

\begin{lemma}
	Any arc $(u,v)$ of $G$ such that $u \not = o(v)$ lies in a unique $(k+1)$-cycle.
\end{lemma}
\begin{proof}
	Let $(u,v)$ be such an arc. As $u \not = o(v)$ there is a path of length $k$ in $G$ from $v$ to $u$, so the arc $(u,v)$ is contained in a $(k+1)$-cycle, which is unique by Lemma \ref{no more than one k-cycle}.
\end{proof}

\begin{corollary}\label{divisibility corollary}
	Suppose that every vertex of $G$ is Type I.  Then $(k+1)$ divides $d(M(d,k)+1) = 2d+d^2+d^3+\dots +d^{k+1}$.
\end{corollary}
\begin{proof}
	As all vertices of $G$ are Type I, by Lemma \ref{no more than one k-cycle} we can partition the arcs of $G$ into disjoint $(k+1)$-cycles. Therefore the size $m = d(M(d,k)+1)$ of $G$ is divisible by $k+1$.   
\end{proof}

Computer search shows that for $k$ between 2 and 10000, the following values of $d$ and $k$ satisfy this condition:
\newline $d = 3$: $k = 2,20,146,902,1028,6320,7202$,
\newline $d = 4$: $k = 3,7,87,171,472,2647$,
\newline $d = 5$: $k= 4,84,114$,
\newline $d = 6$: $k = 2,3,5,7,11,23,32,51,203,347,1095,3323,3767,4903,9563$,
\newline $d = 7$: $k = 6,76,118,2568$,
\newline $d = 8$: $k = 3,7,9,15,87,463,1171$,
\newline $d = 9$: $k = 2,8,68$,
\newline $d = 10$: $k = 3,4,7,9,15,19,39,79,555,1069,2314,2986,4659$,  
\newline $d = 11$: $k = 10$,
\newline $d = 12$: $k = 2,3,5,7,11,13,23,55,91,163,236,1235,1356$.

We see that the condition in Corollary \ref{divisibility corollary} is quite strong. On the other hand, if $G$ contains a Type II vertex then these vertices group themselves into cycles in a natural way. 

\begin{lemma}\label{Type-II outneighbours}
	Suppose that $G$ contains a Type II vertex $u$.  Then $u$ has a unique Type II out-neighbour, namely $o^-(u)$. 
\end{lemma}
\begin{proof}
	Applying the automorphism $o^-$ to the arc $(o(u),u)$, we see that $(u,o^-(u))$ is also an arc.  By inspection, $o^-(u)$ is a Type II vertex.  Suppose that $u'$ is an arbitrary Type II out-neighbour of $u$.  As $(u,u')$ is an arc, so is $(o(u),o(u'))$.  As $u'$ is Type II, $(o(u'),u')$ is an arc.  We therefore have paths $o(u) \rightarrow u \rightarrow u'$ and $o(u) \rightarrow o(u') \rightarrow u'$, so by $k$-geodecity we must have $o(u') = u$, i.e. $u' = o^-(u)$.
\end{proof}

It follows immediately that in a vertex-transitive $(d,k;+1)$-digraph all vertices must be Type I.

\begin{lemma}\label{all vertices type I if VT}
	If $G$ is vertex-transitive, then every vertex of $G$ is Type I.
\end{lemma}
\begin{proof}
	Suppose that $G$ contains a Type II vertex; by vertex-transitivity, every vertex is Type II.  But this contradicts Lemma \ref{Type-II outneighbours}. 
\end{proof}

In particular, if $G$ is vertex-transitive, then it must satisfy the divisibility condition in Corollary \ref{divisibility corollary}.  
\begin{corollary}\label{divisibility cor}
	Let $G$ be a vertex-transitive $(d,k;+1)$-digraph.  Then $(k+1)$ divides $2d+d^2+d^3+\dots +d^{k+1}$.
\end{corollary}
In fact, we can significantly extend Corollary \ref{divisibility corollary} using the fact that for any vertex $u$ of a vertex-transitive $(d,k;+1)$-digraph the vertex $o^-(u)$ cannot be close to $u$.

\begin{lemma}\label{VT distance to o^-}
	If $G$ is a vertex-transitive $(d,k;+1)$-digraph, then for any vertex $u$ of $G$ we have $d(u,o^-(u)) \geq k$.
\end{lemma}
\begin{proof}
	Suppose that $d(u,o^-(u)) = t \leq k-1$. As $G$ is vertex-transitive, the distance from any vertex $v$ of $G$ to $o^-(v)$ is $t$. Writing $N^+(u) = \{ u_1,u_2,\dots ,u_d\}$, let $o^-(u) \in T_{k-2}(u_1)$. As $d(u_2,o^-(u_2)) = t \leq k-1$, we have $o^-(u_2) \in T(u_2)$. However, as $o^-$ is an automorphism of $G$, there is an arc $o^-(u) \rightarrow o^-(u_2)$, so that $o^-(u_2)$ also lies in $T(u_1)$. As $o^-(u_2)$ appears twice in the Moore tree, $k$-geodecity is violated. Hence $t \geq k$.       
\end{proof}

\begin{theorem}\label{divisibility for VT}
	For $1 \leq t \leq k-1$, let $Z_{k+t}$ be the number of distinct directed $(k+t)$-cycles in a vertex-transitive $(d,k;+1)$-digraph $G$. Then for $2 \leq t \leq k-1$ we have
	\begin{equation}\label{equation:general cycle counting}
		(M(d,k)+1)(d^{t}-d^{t-1}) = Z_{k+t}(k+t).	
	\end{equation}  	
\end{theorem}
\begin{proof}
	Let $u$ be any vertex of $G$ and draw the Moore tree of depth $k$ rooted at $u$. By Lemma \ref{VT distance to o^-} all vertices in $T(u)$ have a path of length $\leq k$ to $u$. For $1 \leq t \leq k-1$, let us say that a vertex $v$ at Level $t$ of the Moore tree is \emph{short} if $d(v,u) \leq k-1$ and \emph{long} if $d(v,u) = k$. 
	
	All vertices of $N^+(u)$ must be long by $k$-geodecity. It is easily seen that for $1 \leq t \leq k-2$ every vertex at Level $t$ has one short out-neighbour and $d-1$ long out-neighbours at Level $t+1$. By induction for $2 \leq t \leq k-1$ there are $d^{t-1}$ short vertices in Level $t$ of the tree. Therefore for $2 \leq t \leq k-1$ the vertex $u$ is contained in $d^{t}-d^{t-1}$ closed walks of length $k+t$ and these walks must be cycles by $k$-geodecity. Equation \ref{equation:general cycle counting} follows by double-counting pairs $(u,Z)$, where $u$ is a vertex of $G$ and $Z$ is a directed $(k+t)$-cycle containing $u$.      	
\end{proof}

Using the strengthened divisibility condition in Theorem \ref{divisibility for VT}, computer search shows that the only values of $d$ and $k$ in the range $3 \leq d \leq 12$ and $2 \leq k \leq 10000$ for which there can exist a vertex-transitive $(d,k;+1)$-digraph are 
\newline $k = 2$: $d = 3,6,9,12$,
\newline $k = 3$: $d = 6,10$.
\newline The results of \cite{MilMirSil} show that in practice there are no $(d,3;+1)$-digraphs for $d \geq 2$ and no $(d,2;+1)$-digraphs for $d \geq 8$, so the only remaining values of $d$ and $k$ in this range are $(d,k) = (3,2)$ and $(6,2)$. In Section \ref{(d,2;+1)-digraphs} we will see that such digraphs also do not exist. This scarcity of vertex-transitive $(d,k;+1)$-digraphs can be taken as evidence in favour of Conjecture \ref{no digraphs with excess one}.

There are some simple number-theoretic conditions on $d$ and $k$ that force $G$ to contain a Type II vertex, so that $G$ cannot be vertex-transitive.
\begin{theorem}\label{general conditions for Type II vertex}
	If $d \geq 3$ and $k \geq 2$ satisfy any of the following conditions then any $(d,k;+1)$-digraph contains a Type II vertex:
	\begin{itemize}
		\item i) $d$ and $k$ are odd,
		\item ii) $d \equiv 1 \pmod{k + 1}$ or $d \equiv -1 \pmod{k + 1}$, 
		\item iii) $d^2 | (k+1)$, or
		\item iv) there is an odd prime $p$ such that $p | (k+1)$ and $d \equiv 2 \pmod{p}$.
	\end{itemize}
\end{theorem}
\begin{proof}
	For part i), suppose that $d$ and $k$ are odd, but $G$ contains only Type I vertices. By Corollary \ref{divisibility corollary}, $(k+1)$ divides $d(M(d,k)+1)$, so $M(d,k)$ must be odd. However $M(d,k) = 1+d+d^2+\dots +d^k$ contains an even number of odd summands and hence is even. 
	
	For part ii), suppose that $d \equiv 1 \pmod{k+1}$. Then \[ d(M(d,k)+1) =  d(2+d+d^2+\dots +d^{k}) \equiv k+2 \equiv 1 \pmod{k+1},\] so that $(k+1)$ does not divide $d(M(d,k)+1)$. Similarly if $d \equiv -1 \pmod{k+1}$, then $d(M(d,k)+1) \equiv -2+1-1+1-\dots +(-1)^{k+1} \equiv -2 \pmod{k+1}$ if $k$ is even and $d(M(d,k)+1) \equiv -1 \pmod{k+1}$ if $k$ is odd. 
	
	For part iii), if $d^2$ divides $(k+1)$ and $(k+1)$ divides $d(M(d,k)+1)$, then $d$ divides $M(d,k)+1$, which implies that $ d = 2$. However, we know that there are no $(2,k;+1)$-digraphs \cite{Sil}. 
	
	Finally, for part iv) suppose that $p$ is an odd prime such that $p|(k+1)$ and $d \equiv 2 \pmod{p}$; then if every vertex is Type I we must have \[ 0 \equiv M(d,k)+1 \equiv M(2,k)+1 = 2^{k+1} \pmod{p}, \]	
	implying that $p$ is even, a contradiction.
\end{proof}
These results can be extended to larger values of the excess $\epsilon $ under the assumption of arc-transitivity.

\begin{lemma}\label{arc-transitive lemma}
	Let $\epsilon < d$ and let $G$ be an arc-transitive $(d,k;+\epsilon )$-digraph.  Then for all $u \in V(G)$ we have $O(u) = O^-(u)$.
\end{lemma}
\begin{proof}
	Write $N^+(u) = \{ u_1,u_2,\dots ,u_d\} $.  Suppose that an element $v$ of $O^-(u)$ lies at distance $t \leq k$ from $u$.  We can assume that $v \in T(u_1)$.  By arc-transitivity, there are $d-1$ automorphisms $\phi _i$ of $G$, $i = 2,3,\dots ,d$, that map the arc $u \rightarrow u_1$ to the arc $u \rightarrow u_i$.  Each image $\phi _i(v)$ of $v$ under these automorphisms must also belong to $O^-(u)$, so each branch contains an element of $O^-(u)$.  As $\epsilon < d$, one of these elements must be repeated in the Moore tree rooted at $u$, violating $k$-geodecity.  Therefore we must have $O^-(u) \subseteq O(u)$ and the result follows.
\end{proof}

\begin{corollary}\label{AT divisibility}
	If $G$ is an arc-transitive $(d,k;+\epsilon )$-digraph with excess $\epsilon < d$, then $(k+1)$ divides $d(M(d,k)+\epsilon )$ and $(k+t)$ divides $(M(d,k)+\epsilon )(d^{t}-d^{t-1})$ for $2 \leq t \leq k$. 
\end{corollary}
\begin{proof}
	Again we denote the number of distinct directed $r$-cycles in $G$ by $Z_r$. Using the same reasoning as in the proof of Theorem \ref{divisibility for VT}, we see that each vertex of $G$ is contained in $d$ directed $(k+1)$-cycles, so that $d(M(d,k)+\epsilon )=Z_{k+1}(k+1)$, and each vertex is contained in $d^t-d^{t-1}$ directed $(k+t)$-cycles for $2 \leq t \leq k$, so that $(M(d,k)+\epsilon )(d^t-d^{t-1}) = Z_{k+t}(k+t)$.  
\end{proof}
These results may suggest that the permutation digraphs are smallest possible arc-transitive $k$-geodetic digraphs \cite{TuiErs}.
\section{Diregular digraphs with excess one and degree three}\label{section:degree three}

Sillasen has shown that there are no $(2,k;+1)$-digraphs \cite{Sil}; therefore a reasonable next step is to ask whether there are any $(3,k;+1)$-digraphs. It was proven in \cite{BasMilSirSut} that there are no $(3,k;-1)$-digraphs; the strategy of the proof is to show that any two distinct vertices of a $(3,k;-1)$-digraph can have at most one common out-neighbour (and, conversely, at most one common in-neighbour), classify vertices $u$ according to the distance from $u$ to its repeat $r(u)$ and then count the different types of vertices in two different ways to arrive at a contradiction.  In this section we show that the first main result of \cite{BasMilSirSut}, that any two vertices have at most one common out-neighbour, continues to hold in the setting of digraphs with degree three and excess one. 

The fact that a $(d,k;+1)$-digraph $G$ must be diregular leads to a `duality principle'. This phenomenon has been observed for almost Moore digraphs; in \cite{BasMilSirSut} it is shown that taking the converse of a $(d,k;-1)$-digraph yields another $(d,k;-1)$-digraph. The proof for digraphs with excess one is trivial and we omit it. The Duality Principle will allow us to interchange out-neighbourhoods and in-neighbourhoods in our results. Recall that $o^-$ is the inverse function of the outlier function, i.e. $o^-(v) = u$ if and only if $v = o(u)$.

\begin{lemma}[Duality Principle]\label{Duality Principle}
	Let $G$ be a $(d,k;+1)$-digraph with outlier function $o$. Taking the converse of $G$ yields another $(d,k;+1)$-digraph $G^-$ with outlier function $o^-$.	
\end{lemma}

We begin with a lemma that holds generally for $(d,k;+1)$-digraphs; it describes the situation of vertices with identical out-neighbourhoods.

\begin{lemma}\label{identical neighbourhoods excess one}
	Let $u$ and $v$ be vertices of a $(d,k;+1)$-digraph such that $N^+(u)=N^+(v)$, where $u \not = v$.  Then $v = o(u)$ and $u = o(v)$, i.e. the outlier function $o$ transposes $u$ and $v$. The same result holds if $N^-(u) = N^-(v)$.
\end{lemma}
\begin{proof}
	Suppose that $N^+(u)=N^+(v)$, but $u \not = v$. Draw the Moore tree of depth $k$ rooted at $u$; by $k$-geodecity, $u$ appears only at Level 0, the root position, of this tree. As the Moore tree rooted at $v$ differs from the Moore tree rooted at $u$ only at Level 0, and $u \not = v$, it follows that $v$ cannot reach $u$ by a path of length $\leq k$, so $o(v) = u$ and, by symmetry, $o(u) = v$.  The result for in-neighbourhoods follows by the Duality Principle.
\end{proof}

For the remainder of this section let $G$ be a diregular $(3,k;+1)$-digraph with outlier function $o$. Our first goal is to show that no pair of distinct vertices can have identical out-neighbourhoods; to achieve this we need a lemma for pairs of vertices with exactly two common out-neighbours.

\begin{lemma}\label{two common out-neighbours}
	Let $u,v$ be distinct vertices of $G$ with exactly two common out-neighbours, i.e. $|N^+(u) \cap N^+(v)| = 2$. If we write $N^+(u) = \{ u_1,u_2,u_3\} $ and $N^+(v) = \{ v_1,v_2,v_3\} $, where $u_1 = v_1$, $u_2 = v_2$ and $u_3 \not = v_3$, then $o(u) = v_3$ and $o(v) = u_3$.  
\end{lemma}
\begin{proof}
	Let $u$ and $v$ be as described. This configuration is shown in Figure \ref{fig:two common ONs}. By $k$-geodecity $u_3 \not \in T(u_1) \cup T(u_2)$. Hence there are three possible positions for the vertex $u_3$ in the Moore tree rooted at $v$: i) $u_3 = v$, ii) $u_3 \in T(v_3) - \{ v_3\} $ or iii) $u_3 = o(v)$. If $u_3 = v$, then we have paths $u \rightarrow u_2$ and $u \rightarrow u_3 \rightarrow u_2$, which is impossible for $k \geq 2$.
	
	Suppose that $u_3 \in T(v_3)-\{ v_3\} $.  Put $\ell = d(v_3,u_3)$, so that $1 \leq \ell \leq k-1$.  Let $w$ be a vertex in $N^{k-1-\ell }(u_3)$; then $w \in N^{k-1}(v_3)$.  The vertex $w$ has three out-neighbours $w_1$, $w_2$ and $w_3$.  By $k$-geodecity none of these out-neighbours can lie in $T(v_3)$.  At most two of the out-neighbours can lie in $\{ v,o(v)\} $, so it follows that $w$ has an out-neighbour, say $w_1$, that lies in $T(v_1) \cup T(v_2) = T(u_1) \cup T(u_2)$; without loss of generality $w_1 \in T(u_1)$. Hence there is a path of length $\leq k$ from $u$ to $w_1$ via $u_1$. There is also a path from $u$ to $w_1$ with length $\leq k$ formed from the arc $u \rightarrow u_3$, followed by the path from $u_3$ to $w$ with length $k-1-\ell $ and the arc $w \rightarrow w_1$. This violates $k$-geodecity. It follows that option iii) must hold, i.e. $u_3 = o(v)$. Similarly $v_3 = o(u)$.  
\end{proof}

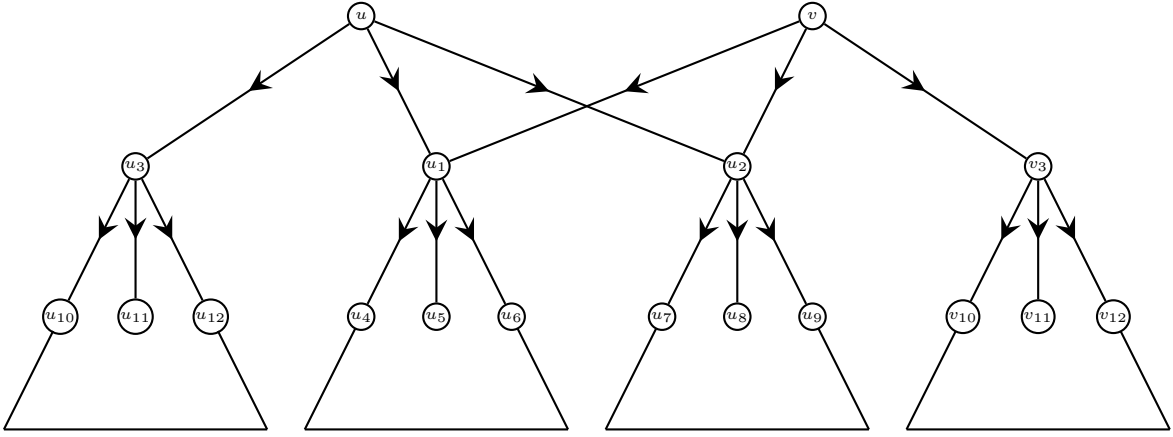
\begin{figure}\centering
	\begin{tikzpicture}[midarrow=stealth,x=0.2mm,y=-0.2mm,inner sep=0.1mm,scale=5,
		thick,vertex/.style={circle,draw,minimum size=10,font=\tiny,fill=white},edge label/.style={fill=white}]
		\tiny

		\node at (-50,0) [vertex] (v1) {$u$};
		\node at (10,0) [vertex] (w1) {$v$};
		
		\node at (-80,20) [vertex] (w2) {$u_3$};
		\node at (-40,20) [vertex] (v2) {$u_1$};
		\node at (0,20) [vertex] (v3) {$u_2$};
		\node at (40,20) [vertex] (v4) {$v_3$};
		
		\node at (-50,40) [vertex] (v5) {$u_4$};
		\node at (-40,40) [vertex] (v11) {$u_5$};
		\node at (-30,40) [vertex] (v6) {$u_6$};
		\node at (-10,40) [vertex] (v7) {$u_7$};
		\node at (0,40) [vertex] (v12) {$u_8$};
		\node at (10,40) [vertex] (v8) {$u_9$};
		\node at (30,40) [vertex] (v9) {$v_{10}$};
		\node at (40,40) [vertex] (v13) {$v_{11}$};
		\node at (50,40) [vertex] (v10) {$v_{12}$};
		\node at (-90,40) [vertex] (w14) {$u_{10}$};
		\node at (-80,40) [vertex] (w15) {$u_{11}$};
		\node at (-70,40) [vertex] (w16) {$u_{12}$};
		
		\path
		
		(w2) edge[middlearrow] (w14)
		(w2) edge[middlearrow] (w15)
		(w2) edge[middlearrow] (w16)
		(v1) edge[middlearrow] (v2)
		(v1) edge[middlearrow] (v3)
		(v1) edge[middlearrow] (w2)

		(w1) edge[middlearrow] (v2)
		(w1) edge[middlearrow] (v3)
		(w1) edge[middlearrow] (v4)	
		
		(v2) edge[middlearrow] (v5)
		(v2) edge[middlearrow] (v6)
		(v2) edge[middlearrow] (v11)
		(v3) edge[middlearrow] (v7)
		(v3) edge[middlearrow] (v8)
		(v3) edge[middlearrow] (v12)
		(v4) edge[middlearrow] (v9)
		(v4) edge[middlearrow] (v10)	
		(v4) edge[middlearrow] (v13);
		
		\draw (-50.75,41.5) -- (-57.5,55);
		\draw (-29.25,41.5) -- (-22.5,55);	
		\draw (-57.5,55)--(-22.5,55);
		
		\draw (-10.75,41.5) -- (-17.5,55);
		\draw (10.75,41.5) -- (17.5,55);	
		\draw (-17.5,55)--(17.5,55);
		
		\draw (29.05,41.9) -- (22.5,55);
		\draw (50.95,41.9) -- (57.5,55);	
		\draw (57.5,55)--(22.5,55);	
		
		\draw (-90.95,41.9) -- (-97.5,55);
		\draw (-69.05,41.9) -- (-62.5,55);	
		\draw (-97.5,55)--(-62.5,55);	
		
	\end{tikzpicture} 
	\caption{Vertices with two common out-neighbours}
	\label{fig:two common ONs}
\end{figure}

\begin{corollary}\label{at most 2 common neighbours}
	No pair $u,v$ of distinct vertices of $G$ have identical out-neighbourhoods.
\end{corollary}
\begin{proof}
	Suppose that $u \not = v$ but $N^+(u) = N^+(v) = \{ u_1,u_2,u_3\} $. The setup is shown in Figure \ref{fig:three common ONs}.  By Lemma \ref{identical neighbourhoods excess one} we know that $v = o(u)$ and $u = o(v)$.  For $i = 1,2,3$ denote the in-neighbour of $u_i$ that does not lie in $\{ u,v\} $ by $u_i^*$. We cannot have $u_1^*=u_2^*=u_3^*$, for otherwise by Lemma \ref{identical neighbourhoods excess one} we would have $o(u) = u_1^* = v$.
	
	By the Duality Principle taking the converse $G^-$ of $G$ yields a diregular $(3,k;+1)$-digraph with outlier function $o' = o^-$. In $G^-$ we have $N^+(u_i) = \{ u_i^*,u,v\} $ for $i = 1,2,3$. 
	
	Suppose that $u_i^* \not = u_j^*$. Then in $G^-$ the pair of vertices $u_i,u_j$ has exactly two common out-neighbours, so that by Lemma \ref{two common out-neighbours} we obtain $o^-(u_i) = u_j^*$ and $o^-(u_j) = u_i^*$. If $u_1^*,u_2^*,u_3^*$ are all distinct, we would then obtain $o^-(u_1) = u_2^*=u_3^*$, a contradiction. 
	
	We can thus assume that $u_1^* = u_2^* \not = u_3^*$. Applying Lemma \ref{two common out-neighbours} to the pairs $u_1,u_3$ and $u_2,u_3$ in turn, we deduce that $o^-(u_1) = o^-(u_2) = u_3^*$, again a contradiction, as $o$ is a permutation.   
\end{proof}

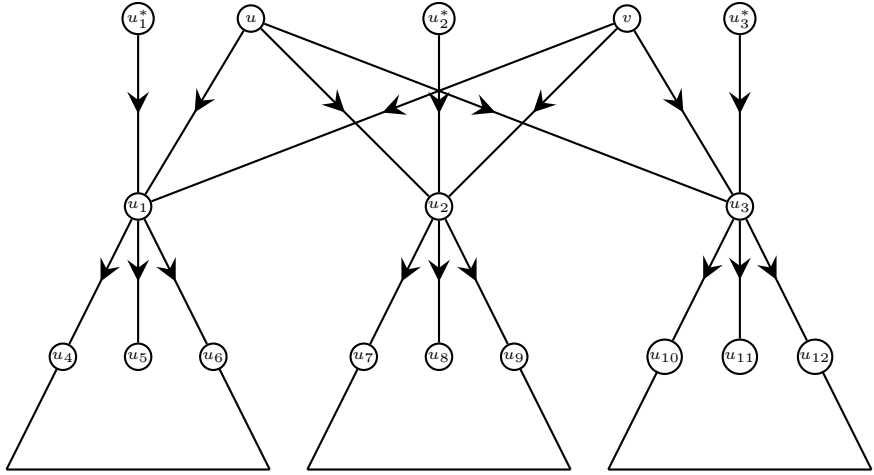
\begin{figure}\centering
	\begin{tikzpicture}[midarrow=stealth,x=0.2mm,y=-0.2mm,inner sep=0.1mm,scale=5,
		thick,vertex/.style={circle,draw,minimum size=10,font=\tiny,fill=white},edge label/.style={fill=white}]
		\tiny

		\node at (-25,-5) [vertex] (v1) {$u$};
		\node at (25,-5) [vertex] (w1) {$v$};
		
		\node at (-40,-5) [vertex] (u1*) {$u_1^*$};
		\node at (0,-5) [vertex] (u2*) {$u_2^*$};
		\node at (40,-5) [vertex] (u3*) {$u_3^*$};
		
		\node at (-40,20) [vertex] (v2) {$u_1$};
		\node at (0,20) [vertex] (v3) {$u_2$};
		\node at (40,20) [vertex] (v4) {$u_3$};
		
		\node at (-50,40) [vertex] (v5) {$u_4$};
		\node at (-40,40) [vertex] (v11) {$u_5$};
		\node at (-30,40) [vertex] (v6) {$u_6$};
		\node at (-10,40) [vertex] (v7) {$u_7$};
		\node at (0,40) [vertex] (v12) {$u_8$};
		\node at (10,40) [vertex] (v8) {$u_9$};
		\node at (30,40) [vertex] (v9) {$u_{10}$};
		\node at (40,40) [vertex] (v13) {$u_{11}$};
		\node at (50,40) [vertex] (v10) {$u_{12}$};
		
		\path
		(u1*) edge[middlearrow] (v2)
		(u2*) edge[middlearrow] (v3)
		(u3*) edge[middlearrow] (v4)
		
		(v1) edge[middlearrow] (v2)
		(v1) edge[middlearrow] (v3)
		(v1) edge[middlearrow] (v4)

		(w1) edge[middlearrow] (v2)
		(w1) edge[middlearrow] (v3)
		(w1) edge[middlearrow] (v4)	
		
		(v2) edge[middlearrow] (v5)
		(v2) edge[middlearrow] (v6)
		(v2) edge[middlearrow] (v11)
		(v3) edge[middlearrow] (v7)
		(v3) edge[middlearrow] (v8)
		(v3) edge[middlearrow] (v12)
		(v4) edge[middlearrow] (v9)
		(v4) edge[middlearrow] (v10)	
		(v4) edge[middlearrow] (v13);
		
		\draw (-50.75,41.5) -- (-57.5,55);
		\draw (-29.25,41.5) -- (-22.5,55);	
		\draw (-57.5,55)--(-22.5,55);
		
		\draw (-10.75,41.5) -- (-17.5,55);
		\draw (10.75,41.5) -- (17.5,55);	
		\draw (-17.5,55)--(17.5,55);
		
		\draw (29.05,41.9) -- (22.5,55);
		\draw (50.95,41.9) -- (57.5,55);	
		\draw (57.5,55)--(22.5,55);	
		
	\end{tikzpicture} 
	\caption{Configuration for Corollary \ref{at most 2 common neighbours}}
	\label{fig:three common ONs}
\end{figure}

Having ruled out identical out-neighbourhoods, we can complete the proof of our desired result.
\begin{theorem}\label{one common ON theorem}
	Any two distinct vertices of a $(3,k;+1)$-digraph $G$ have at most one common out-neighbour and at most one common in-neighbour.
\end{theorem}
\begin{proof}
	Suppose that $u,v$ are distinct vertices with more than one out-neighbour in common. By Corollary \ref{at most 2 common neighbours}, $u$ and $v$ must have exactly two common out-neighbours. Write $N^+(u) = \{ u_1,u_2,u_3\} $ and $N^+(v) = \{ v_1,v_2,v_3\} $, where $u_1 = v_1, u_2 = v_2$, but $u_3 \not = v_3$ (see Figure \ref{fig:two common ONs}).  By Lemma \ref{two common out-neighbours} we know that $o(u) = v_3$ and $o(v) = u_3$.

	Let $w \in T(u_3) \cap T(v_3)$, with $d(u_3,w) = s$ and $d(v_3,w) = t$.  Suppose that $s > t$.  Consider the set $N^{k-s}(w)$.  By construction, $N^{k-s}(w) \subseteq N^k(u_3)$, so $N^{k-s}(w) \cap T(u_3) = \emptyset $.  We have $k+t-s \leq k-1$, so $N^{k-s}(w) \subseteq T(v_3)$. Hence by $k$-geodecity $N^{k-s}(w) \cap (T(u_1) \cup T(u_2)) = \emptyset $.  As no vertex of $N^{k-s}(w)$ can lie in any of the branches of the Moore tree rooted at $u$, we must have $N^{k-s}(w) \subseteq \{ u,o(u)\} $. Thus the size of the set $N^{k-s}(w)$ satisfies $|N^{k-s}(w)| = 3^{k-s} \leq 2$, which is impossible for $s \leq k-1$. Therefore $d(u_3,w) = d(v_3,w)$ for every $w \in T(u_3)\cap T(v_3)$.
	
	Consider $N^+(u_3)$ and $N^+(v_3)$. By $k$-geodecity $N^+(u_3) \cap (T(u_1) \cup T(u_2)) = \emptyset $. Also $v_3 \not \in N^+(u_3)$, as $v_3 = o(u)$, and $o(v) = u_3 \not \in N^+(u_3)$. Thus $N^+(u_3) \subset \{ v\} \cup N^+(v_3)$ and similarly $N^+(v_3) \subset \{ u\} \cup N^+(u_3)$. By Corollary \ref{at most 2 common neighbours} we cannot have $N^+(u_3) = N^+(v_3)$, so we can assume that $u_{10} = v_{10}$, $u_{11}=v_{11}$, $u_{12} = v$ and $v_{12} = u$. If $k \geq 3$ then $u$ will have distinct $\leq k$-paths to $u_1$ (and $u_2$), namely $u \rightarrow u_1$ and $u \rightarrow u_3 \rightarrow v \rightarrow u_1$, so $k = 2$. The resulting configuration is displayed in Figure \ref{fig:two common ONsv2}.
	
	Observe that now $u_3$ and $v_3$ have two out-neighbours in common, namely $u_{10}$ and $u_{11}$, so by Lemma \ref{two common out-neighbours} we have $o(u_3) = u$ and $o(v_3) = v$. Applying the outlier automorphism to the arcs incident with $u$, we deduce that $o(u) = v_3$ has arcs to $o(u_3) = u$ and $o(u_1)$ and $o(u_2)$, so $\{ o(u_1),o(u_2)\} = \{ u_{10},u_{11}\} $. By $2$-geodecity $u_{10}$ can have arcs only to $N^+(u_1)$ and $N^+(u_2)$. As $u_{10}$ has three out-going arcs and cannot have the same out-neighbourhood as $u_1$ or $u_2$ by Corollary \ref{at most 2 common neighbours}, it follows that $u_{10}$ must have two common out-neighbours with either $u_1$ or $u_2$; without loss of generality $N^+(u_{10}) = \{ u_4,u_5,u_7\} $. Applying Lemma \ref{two common out-neighbours} to the pair $u_1,u_{10}$ we see that $o(u_1) = u_7$. As we have already determined that $o(u_1)\in \{ u_{10},u_{11}\} $, this is a contradiction. 
	The last part of the theorem follows by the Duality Principle.	
\end{proof}
\begin{figure}\centering
	\begin{tikzpicture}[midarrow=stealth,x=0.2mm,y=-0.2mm,inner sep=0.1mm,scale=5,
		thick,vertex/.style={circle,draw,minimum size=10,font=\tiny,fill=white},edge label/.style={fill=white}]
		\tiny

		\node at (-50,0) [vertex] (v1) {$u$};
		\node at (10,0) [vertex] (w1) {$v$};
		
		\node at (-80,20) [vertex] (w2) {$u_3$};
		\node at (-40,20) [vertex] (v2) {$u_1$};
		\node at (0,20) [vertex] (v3) {$u_2$};
		\node at (40,20) [vertex] (v4) {$v_3$};
		
		\node at (-50,40) [vertex] (v5) {$u_4$};
		\node at (-40,40) [vertex] (v11) {$u_5$};
		\node at (-30,40) [vertex] (v6) {$u_6$};
		\node at (-10,40) [vertex] (v7) {$u_7$};
		\node at (0,40) [vertex] (v12) {$u_8$};
		\node at (10,40) [vertex] (v8) {$u_9$};
		\node at (30,40) [vertex] (v9) {$u_{10}$};
		\node at (40,40) [vertex] (v13) {$u_{11}$};
		\node at (50,40) [vertex] (v10) {$u$};
		\node at (-90,40) [vertex] (w14) {$u_{10}$};
		\node at (-80,40) [vertex] (w15) {$u_{11}$};
		\node at (-70,40) [vertex] (w16) {$v$};
		
		\path
		
		(w2) edge[middlearrow] (w14)
		(w2) edge[middlearrow] (w15)
		(w2) edge[middlearrow] (w16)
		(v1) edge[middlearrow] (v2)
		(v1) edge[middlearrow] (v3)
		(v1) edge[middlearrow] (w2)

		(w1) edge[middlearrow] (v2)
		(w1) edge[middlearrow] (v3)
		(w1) edge[middlearrow] (v4)	
		
		(v2) edge[middlearrow] (v5)
		(v2) edge[middlearrow] (v6)
		(v2) edge[middlearrow] (v11)
		(v3) edge[middlearrow] (v7)
		(v3) edge[middlearrow] (v8)
		(v3) edge[middlearrow] (v12)
		(v4) edge[middlearrow] (v9)
		(v4) edge[middlearrow] (v10)	
		(v4) edge[middlearrow] (v13);
		
		\draw (-50.75,41.5) -- (-57.5,55);
		\draw (-29.25,41.5) -- (-22.5,55);	
		\draw (-57.5,55)--(-22.5,55);
		
		\draw (-10.75,41.5) -- (-17.5,55);
		\draw (10.75,41.5) -- (17.5,55);	
		\draw (-17.5,55)--(17.5,55);
		
		\draw (29.05,41.9) -- (22.5,55);
		\draw (50.95,41.9) -- (57.5,55);	
		\draw (57.5,55)--(22.5,55);	
		
		\draw (-90.95,41.9) -- (-97.5,55);
		\draw (-69.05,41.9) -- (-62.5,55);	
		\draw (-97.5,55)--(-62.5,55);	
		
	\end{tikzpicture} 
	\caption{Configuration for Theorem \ref{one common ON theorem} }
	\label{fig:two common ONsv2}
\end{figure}
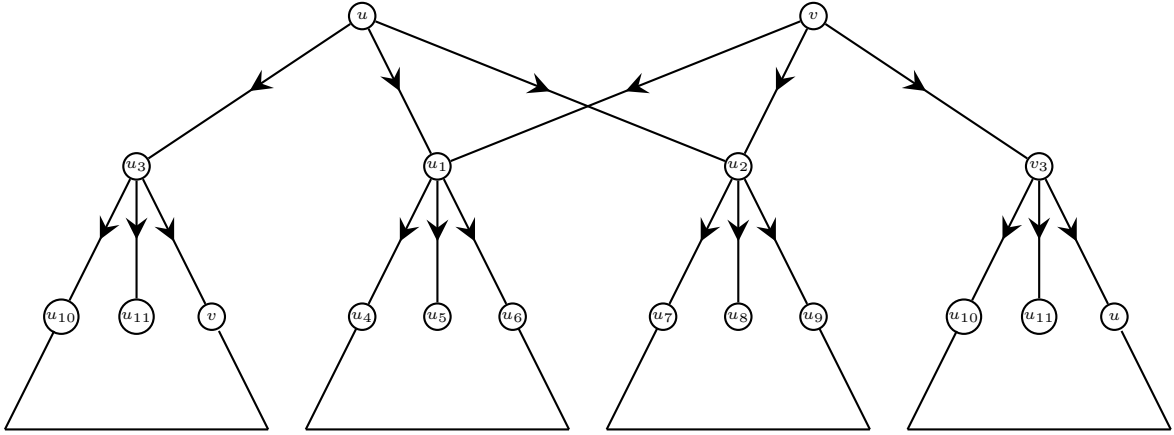

Theorem \ref{one common ON theorem} allows us to prove our first non-existence result for degree three, namely that there are no $(3,2;+1)$-digraphs.

\begin{theorem}\label{no (3,2;+1)-digraphs}
	There are no $(3,2;+1)$-digraphs.	
\end{theorem}
\begin{proof}
	Suppose that $G$ is a diregular $(3,2;+1)$-digraph. Fix an arbitrary vertex $u$ of $G$ with $N^+(u) = \{ u_1,u_2,u_3\} $ and draw the Moore tree rooted at $u$ as shown in Figure \ref{fig:Moore tree d = 3 k = 2 excess 1}. We set $N^-(u_1) = \{ u,a_1,a_2\} $, $N^-(u_2) = \{ u,b_1,b_2\} $ and $N^-(u_3) = \{ u,c_1,c_2\} $.
	
	At least one of the vertices $c_1,c_2$ is not equal to $o(u)$, say $c_1 \not = o(u)$. By $2$-geodecity we can assume that $c_1 = u_4$. By Theorem \ref{one common ON theorem} $c_1$ has no arcs to $T(u)-\{ u_3\}$, at most one arc to $N^+(u_2)$ and by $2$-geodecity has no arcs to $N^+(u_1) \cup N^+(u_3)$. It follows that $c_1$ must have exactly one arc to $N^+(u_2)$ as well as an arc to $o(u)$. If $c_2 = o(u)$ this would yield two paths of length $\leq 2$ from $c_1$ to $u_3$, so $c_2 \not = o(u)$. By the same reasoning $c_2$ has an arc to $o(u)$; however, we now have two distinct vertices with at least two common out-neighbours, contradicting Theorem \ref{one common ON theorem}. 
\end{proof}

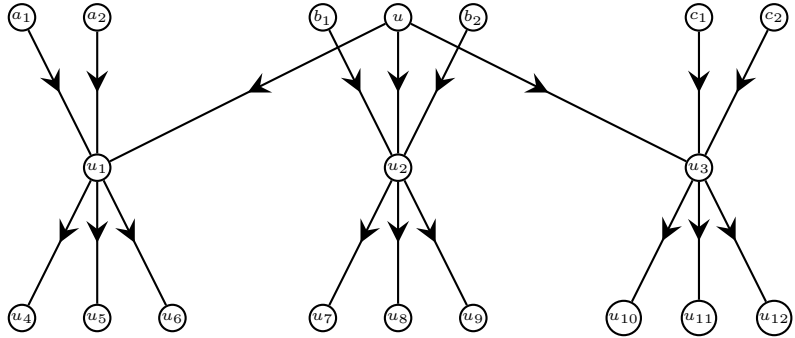
\begin{figure}\centering
	\begin{tikzpicture}[midarrow=stealth,x=0.2mm,y=-0.2mm,inner sep=0.1mm,scale=5,
		thick,vertex/.style={circle,draw,minimum size=10,font=\tiny,fill=white},edge label/.style={fill=white}]
		\tiny
		
		\node at (-50,0) [vertex] (a1) {$a_1$};
		\node at (-40,0) [vertex] (a2) {$a_2$};
		
		\node at (-10,0) [vertex] (b1) {$b_1$};
		\node at (10,0) [vertex] (b2) {$b_2$};
		
		\node at (40,0) [vertex] (c1) {$c_1$};
		\node at (50,0) [vertex] (c2) {$c_2$};
		
		\node at (0,0) [vertex] (v1) {$u$};
		
		\node at (-40,20) [vertex] (v2) {$u_1$};
		\node at (0,20) [vertex] (v3) {$u_2$};
		\node at (40,20) [vertex] (v4) {$u_3$};
		
		\node at (-50,40) [vertex] (v5) {$u_4$};
		\node at (-40,40) [vertex] (v11) {$u_5$};
		\node at (-30,40) [vertex] (v6) {$u_6$};
		\node at (-10,40) [vertex] (v7) {$u_7$};
		\node at (0,40) [vertex] (v12) {$u_8$};
		\node at (10,40) [vertex] (v8) {$u_9$};
		\node at (30,40) [vertex] (v9) {$u_{10}$};
		\node at (40,40) [vertex] (v13) {$u_{11}$};
		\node at (50,40) [vertex] (v10) {$u_{12}$};
		
		\path
		(a1) edge[middlearrow] (v2)
		(a2) edge[middlearrow] (v2)
		(b1) edge[middlearrow] (v3)
		(b2) edge[middlearrow] (v3)
		(c1) edge[middlearrow] (v4)
		(c2) edge[middlearrow] (v4)
		(v1) edge[middlearrow] (v2)
		(v1) edge[middlearrow] (v3)
		(v1) edge[middlearrow] (v4)	
		
		(v2) edge[middlearrow] (v5)
		(v2) edge[middlearrow] (v6)
		(v2) edge[middlearrow] (v11)
		(v3) edge[middlearrow] (v7)
		(v3) edge[middlearrow] (v8)
		(v3) edge[middlearrow] (v12)
		(v4) edge[middlearrow] (v9)
		(v4) edge[middlearrow] (v10)	
		(v4) edge[middlearrow] (v13);	
		
	\end{tikzpicture} 
	\caption{The Moore tree for Theorem \ref{no (3,2;+1)-digraphs}}
	\label{fig:Moore tree d = 3 k = 2 excess 1}
\end{figure}

\section{Automorphisms of digraphs with excess one}\label{Automorphisms}

In \cite{Sil2} Sillasen uses counting arguments to deduce information on the form of the subdigraph of an almost Moore digraph that is induced by the set of vertices fixed by an automorphism. Using the same approach we can deduce a strong result on the action of automorphisms of a digraph with excess one. This will later help us to analyse the structure of the outlier function of a $(d,k;+1)$-digraph.

Let $G$ be a $(d,k;+1)$-digraph and $\phi  \in \aut(G)$ a non-identity automorphism of $G$.  Denote by $\fix(\phi )$ the set of vertices of $G$ that are fixed by $\phi $ and let $\FIX(\phi )$ be the subdigraph induced by $\fix(\phi )$.

Firstly we show that the fix-set $\fix(\phi )$ is closed under the action of the outlier automorphism.

\begin{lemma}\label{outliers of fixvertices are fixvertices}
	If $u \in \fix(\phi )$, then $o^j(u) \in \fix(\phi )$ for all $j \in \mathbb{N}$.
\end{lemma}
\begin{proof}
	Let $u \in \fix(\phi )$.  If $d(u,v) \leq k$, then $d(\phi (u),\phi (v)) = d(u,\phi (v)) \leq k$, so the only vertex of $V(G)$ that lies at distance $\geq k+1$ from $u$ is $\phi (o(u))$ and so $\phi (o(u)) = o(u)$ and $o(u) \in \fix(\phi )$.  Iteration of $o$ implies the result.
\end{proof}
As the outlier automorphism is fixed-point-free, it follows from Lemma \ref{outliers of fixvertices are fixvertices} that any fix-set $\fix(\phi )$ cannot consist of a single vertex and, if $\phi $ fixes just two vertices $u,u'$ of $G$, then these two vertices are outliers of each other, i.e. $o(u) = u'$ and $o(u') = u$.

\begin{corollary}
	If $|\fix (\phi )| \leq 2$, then either $\fix(\phi ) = \emptyset $ and $\FIX(\phi )$ is the null digraph, or $|\fix(\phi )| = 2$ and $\FIX(\phi ) \cong 2K_1$.
\end{corollary}

We will now assume that $\fix(\phi )$ contains at least three vertices.

\begin{lemma}\label{leq k paths contained in Fix}
	If $u,v \in \fix(\phi )$ and $P$ is a path of length $\leq k$ from $u$ to $v$, then all vertices of $P$ are contained in $\fix(\phi )$.
\end{lemma}
\begin{proof}
	Let $u, v$ and $P$ be as described.  Suppose that there is a vertex $u' \in V(P)$ that is not fixed by $\phi $.  Then $P$ and $\phi (P)$ are distinct $\leq k$-paths from $u$ to $v$, contradicting $k$-geodecity.
\end{proof}

\begin{lemma}\label{FIX diregular}
	The digraph $\FIX(\phi )$ is diregular.
\end{lemma}
\begin{proof}
	For any vertex $u \in \fix(\phi )$ we will denote the out-degree and in-degree of $u$ in the subdigraph $\FIX(\phi )$ by $d_{\phi }^+(u)$ and $d_{\phi }^-(u)$ respectively. We will show that for any two (not necessarily distinct) vertices $u,v \in \fix(\phi )$ we have $d_{\phi }^+(u) = d_{\phi }^-(v)$; this implies the desired result. For this pair $u,v$ we will write $N^+(u) = \{ u_1,u_2,\dots ,u_d\} $ and $N^-(v) = \{ v_1,\dots ,v_d\} $.
	
	Assume that $v \not \in o(N^+(u))$.  Then for $1 \leq i \leq d$ there is a unique $\leq k$-path $P_i$ from $u_i$ to $v$. Suppose that $u \not \rightarrow v$.  By $k$-geodecity, none of the paths $P_i$ pass through the same in-neighbour of $v$, so without loss of generality there is a $\leq (k-1)$-path from $u_i$ to $v_i$ for $1 \leq i \leq d$.  By Lemma \ref{leq k paths contained in Fix}, it follows that $u_i \in \fix(\phi )$ if and only if $v_i \in \fix(\phi )$, so that $d^+_{\phi }(u) = d^-_{\phi }(v)$. If $u \rightarrow v$, then repeating this reasoning for the out-neighbours of $u$ other than $v$ shows that we still have $d_{\phi }^+(u) = d_{\phi }^-(v)$.
	
	Now suppose that $v \in o(N^+(u))$; say $v = o(u_1)$.  If $u \not \rightarrow v$, then for $2 \leq i \leq d$ we can assume that there is a $\leq (k-1)$-path from $u_i$ to $v_i$ and as before $u_i \in \fix(\phi )$ if and only if $v_i \in \fix(\phi )$ for $2 \leq i \leq d$.  There is an arc $u \rightarrow u_1$, so as $o$ is an automorphism there exists an arc $o(u) \rightarrow o(u_1) = v$, giving $o(u) \in N^-(v)$.  There are $\leq k$-paths from $u$ to each $v_i$ for $2 \leq i \leq d$, so we must have $o(u) = v_1$.  As $o(u) \in \fix(\phi )$ by Lemma \ref{outliers of fixvertices are fixvertices} we have $v_1 \in \fix(\phi )$.  Also by Lemma \ref{outliers of fixvertices are fixvertices} we have $o^-(v) = u_1 \in \fix(\phi )$, so again we see that $d^+_{\phi }(u) = d^-_{\phi }(v)$. Again the case $u \rightarrow v$ is similar.
	
	It follows that $\FIX(\phi )$ is diregular.
\end{proof}

\begin{lemma}\label{isometric}
	The digraph $\FIX(\phi )$ is an isometric subdigraph of $G$ and has diameter $k+1$.
\end{lemma}
\begin{proof}
	As $\fix(\phi )$ is a subdigraph of $G$ we certainly have $d_{\FIX(\phi )}(u,v) \geq d_G(u,v)$ for all $u,v \in \fix(\phi )$.  Let $u,v \in \FIX(\phi )$ be arbitrary.  If $v \in T_k(u) \cap \fix(\phi )$, then by Lemma \ref{leq k paths contained in Fix} the path from $u$ to $v$ in $G$ also belongs to $\FIX(\phi )$, so that $d_G(u,v) = d_{\FIX(\phi )}(u,v)$.  
	
	By Lemma \ref{FIX diregular}, $\FIX(\phi )$ is diregular with degree $\geq 1$ (as we are assuming that $|\fix(\phi )| \geq 3$). Hence if $v = o(u)$ in $G$, then $o(u)$ has an in-neighbour $v'$ in $\fix(\phi )$, so that by the preceding argument $d_G(u,v') = d_{\FIX(\phi )}(u,v')$ and thus $ d_{\FIX(\phi )}(u,o(u)) = k+1$. Therefore $\FIX(\phi )$ is an isometric subdigraph of $G$ and, since $o(u) \in \fix(\phi )$ for any $u \in \fix(\phi )$, the diameter of $\FIX(\phi )$ is exactly $k+1$.\end{proof}

\begin{corollary}\label{fixset}
	The digraph $\FIX(\phi )$ is a $(d',k;+1)$-digraph for some $d'$ in the range $1 \leq d' \leq d-1$.
\end{corollary}
\begin{proof}
	As a subdigraph of $G$, $\FIX(\phi )$ is $k$-geodetic. By Lemma \ref{FIX diregular}, $\FIX(\phi )$ is diregular with degree $d'$. We are assuming that $\fix(\phi )$ contains at least three vertices, so by Lemma \ref{leq k paths contained in Fix} $\FIX(\phi )$ contains a path and $d' \geq 1$. We are also assuming that $\phi $ is not the identity automorphism, so $\phi $ does not fix all vertices of $G$ and $d' \leq d-1$.  By diregularity and Lemma \ref{isometric}, it follows that $\FIX(\phi )$ has order $M(d',k)+1$, so $\FIX(\phi )$ is a $(d',k;+1)$-digraph.    
\end{proof}

As there are no diregular $(2,k;+1)$-digraphs \cite{Sil2}, we have the following result.

\begin{corollary}\label{automorphism fixset}
	If $G$ is a $(d,k;+1)$-digraph and $\phi $ is a non-identity automorphism of $G$, then $\FIX(\phi )$ is either the null digraph, a pair of isolated vertices, a directed $(k+2)$-cycle or a $(d',k;+1)$-digraph, where $3 \leq d' \leq d-1$.
\end{corollary}

\section{Structure of the outlier function}\label{section:spectrum digraphs excess one}

We will now make use of some of the results from the preceding sections to deduce useful information on the permutation structure of the outlier function of a digraph with excess one. Let $G$ be a $(d,k;+1)$-digraph. By Lemma \ref{lem:known results} $G$ is diregular and the outlier function of $G$ is an automorphism. Therefore every vertex $u$ of $G$ has an associated order $\omega (u)$, which is the smallest integer such that $o^{\omega }(u) = u$. Cholily et al. have used the vertex orders of the repeat function to successfully analyse the structure of digraphs with defect one in such papers as \cite{BasChoMil} and \cite{ChoBasUtt}. We can immediately apply the method of \cite{ChoBasUtt} to make a connection between the vertex orders for the outlier function and the existence of short paths in $G$. 

\begin{lemma}\label{divisibility of orders}
	Let $u_0,u_1,\dots ,u_r$ be a path of length $r$ in $G$, where $r \leq k$, and put $t = \lcm(\omega (u_0),\omega (u_r))$.  Then $\omega (u_i)$ divides $t$ for $1 \leq i \leq r-1$.
\end{lemma}
\begin{proof}
	Suppose that for some $1 \leq i \leq r-1$ the order of $u_i$ does not divide $t$.  Then $o^t(u_i) \not = u_i$, so we obtain two $\leq k$-paths $u_0,u_1,\dots ,u_i,\dots ,u_r$ and $o^t(u_0),o^t(u_1),\dots ,o^t(u_i),\dots ,o^t(u_r) = u_0,o^t(u_1),\dots ,o^t(u_i),\dots ,u_r$ from $u_0$ to $u_r$, a contradiction.
\end{proof}

\begin{corollary}\label{paths between minimal vertices}
	If $p$ is the minimum vertex order of $G$ and $W$ is a walk of length $\leq k$ between two vertices $u, v$ with order $p$, then every vertex on $W$ has order $p$.
\end{corollary}
\begin{proof}
	Suppose that there is a vertex $w$ on $W$ such that $\omega (w) > p$.  Then $W$ and $o^p(W)$ are two distinct walks of length $\leq k$ between $u$ and $v$, contradicting $k$-geodecity.
\end{proof}

We now make two definitions that will help us to analyse the structure of the permutation $o$.

\begin{definition}
	The \emph{index} $\omega (G)$ of a $(d,k;+1)$-digraph $G$ is the value of the smallest vertex order in $G$, i.e. $\omega (G) = \min \{ \omega (u): u \in V(G)\} $.
\end{definition}
\begin{definition}
	A $(d,k;+1)$-digraph is \emph{outlier-regular} if its outlier function $o$ is a regular permutation. If each vertex of $G$ has order $\omega $, then $G$ is \emph{$\omega $-outlier-regular}. 
\end{definition}

As $o$ is an automorphism it follows that any power $o^r$ of $o$ is also an automorphism of $G$. In Section \ref{Automorphisms} we classified the possible fixed sets of any non-identity automorphism of $G$. We therefore record the following implication of Corollary \ref{automorphism fixset}.

\begin{corollary}\label{vertex order corollary}
	For any integer $r \geq 2$, the set of vertices of $G$ with order dividing $r$ induces one of the following:
	\begin{itemize}
		\item the entire digraph $G$,
		\item the empty digraph,
		\item a pair of vertices that form a transposition in $o$,
		\item a directed $(k+2)$-cycle, or
		\item a $(d',k;+1)$-digraph, where $3 \leq d' \leq d-1$.
		
	\end{itemize} 	
\end{corollary}
Conjecture \ref{no digraphs with excess one} claims that there are no non-trivial $(d,k;+1)$-digraphs; one approach to proving this conjecture is to study the properties of a minimal counterexample. Let $k \geq 2$ and $k \not = 3,4$. Suppose that there exists a $(d,k;+1)$-digraph with $d \geq 2$ and let $d'$ be the smallest possible value of $d \geq 3$ such that there exists a $(d,k;+1)$-digraph; then we will refer to a $(d',k;+1)$-digraph as a \emph{minimal $(d,k;+1)$-digraph}. For a fixed $k$, Corollary \ref{vertex order corollary} strongly restricts the structure of the outlier automorphism of a minimal $(d,k;+1)$-digraph.  

\begin{corollary}\label{vertex orders for degree 3}
	A minimal $(d,k;+1)$-digraph $G$ satisfies one of the following:
	\begin{itemize}
		\item $G$ is outlier-regular,
		\item the outlier function $o$ of $G$ contains a unique transposition, or
		\item the vertices of $G$ with order $\omega (G)$ form a directed $(k+2)$-cycle.
	\end{itemize} 
	In particular this holds for any $(3,k;+1)$-digraph. 	
\end{corollary}   
\begin{proof}
	By Corollary \ref{vertex order corollary} the automorphism $o^{\omega (G)} $ fixes either i) every vertex of $G$, in which case every vertex of $G$ has order $\omega (G)$ and $G$ is outlier-regular, ii) two vertices that are outliers of each other, so that $\omega (G) = 2$ and $o$ contains a unique transposition, or iii) a $(k+2)$-cycle. 
\end{proof}

If a minimal $(d,k;+1)$-digraph is not outlier-regular, then Corollary \ref{vertex orders for degree 3} allows us to deduce the subdigraph induced by the set of vertices with smallest order.
\begin{lemma}\label{omega is small}
	If a minimal $(d,k;+1)$-digraph $G$ is not outlier-regular, then either its outlier function $o$ contains a unique transposition, or else $\omega (G) = k+2$ and the vertices with order $k+2$ induce a directed $(k+2)$-cycle. 	
\end{lemma}
\begin{proof}
	Suppose that $G$ is a non-outlier-regular minimal $(d,k;+1)$-digraph with outlier function that does not contain a unique transposition. Then by Corollary \ref{vertex orders for degree 3} the vertices with order equal to the index $\omega (G)$ of $G$ induce a directed $(k+2)$-cycle $C$. For any vertex $u$ in the cycle its outlier $o(u)$ also has order $\omega (G)$, so the outlier of $u$ must be the vertex preceding $u$ on the cycle $C$; it follows that $\omega (G) = k+2$.   
\end{proof}
We will find the following classification of this behaviour convenient.

\begin{definition}
	A minimal $(d,k;+1)$-digraph $G$ such that the vertices of $G$ with order $\omega (G)$ form a directed $(k+2)$-cycle is \emph{Type A}, whereas if the outlier function $o$ of $G$ contains a unique transposition, then $G$ is \emph{Type B}.
\end{definition}
According to this definition, every minimal $(d,k;+1)$-digraph is either Type A, Type B or outlier-regular. Let us now return to the problem of digraphs with degree three and excess one; if any such digraph exists it is minimal.

\begin{lemma}\label{index 3 or more}
	Let $G$ be a $(3,k;+1)$-digraph of Type A. Then $k+2$ divides $\frac{M(3,k)-k-1}{2}$.
\end{lemma}
\begin{proof}
	By Lemma \ref{omega is small} we have $\omega (G) = k+2$ and the vertices with order $k+2$ induce a $(k+2)$-cycle $C$. Pick a vertex $u$ on $C$ and write $N^+(u) = \{ u_1,u_2,u_3\} $, where $u_1$ also lies on $C$. The automorphism $o^{(k+2)}$ fixes $u$ and $u_1$, but not $u_2$ and $u_3$, so $o^{(k+2)}$ transposes $u_2$ and $u_3$. Thus $o^{2(k+2)}$ fixes every vertex in $T_1(u)$ and by Corollary \ref{vertex order corollary} every vertex of $G$ has order either $k+2$ or $2(k+2)$. If there are $r$ cycles in $o$ with length $2(k+2)$, then we obtain \[ M(3,k)+1=k+2+2r(k+2)\]
	and $k+2$ divides $\frac{M(3,k)-k-1}{2}$. 
\end{proof}

\begin{lemma}\label{2 and 6}
	Let $G$ be a $(3,k;+1)$-digraph of Type B. Then $k \not \equiv 3,5\pmod{6}$. If $k \equiv 0,2 \pmod{6}$, then $G$ contains two vertices of order two, with all other vertices of $G$ having order six. 
\end{lemma}	
\begin{proof}
	Assume that $G$ is a non-outlier-regular $(3,k;+1)$-digraph with outlier function $o$ containing a unique transposition. Let $u$ and $o(u)$ be the vertices of $G$ with order two, where $N^+(u) = \{ u_1,u_2,u_3\} $. The automorphism $o^2$ fixes $u$, but fixes no vertex in $\{ u_1,u_2,u_3\} $.  We can thus assume that $o^2$ permutes $u_1,u_2,u_3$ in a 3-cycle $(u_1u_2u_3)$. 
	
	By Theorem \ref{one common ON theorem} $u$ and $o(u)$ have at most one common out-neighbour. Suppose that $N^+(u) \cap N^+(o(u)) \not = \emptyset $; we can assume that $u_1$ is the common out-neighbour of $u$ and $o(u)$. However, applying the automorphism $o^2$ to $G$ shows that $o^2(u_1)$ is a common out-neighbour of $u$ and $o(u)$ and, since $G$ is Type B we have $o^2(u_1) \not = u_1$, thereby violating Theorem \ref{one common ON theorem}. It follows that $o$ contains the 6-cycle $(u_1,o(u_1),u_2,o(u_2),u_3,o(u_3))$. 
	
	Therefore $\{ u\} \cup N^+(u) \subseteq \fix(o^6)$ and hence $o^6$ fixes every vertex of $G$, so that the order of every vertex apart from $u$ and $o(u)$ is either $3$ or $6$. 
	
	Suppose that there is a vertex with order $3$. Then $o^3$ fixes a $(k+2)$-cycle. If there are $r$ cycles in $o$ of length $6$, then \[ M(3,k)+1  = 2+(k+2)+6r.\] Hence $6|(M(3,k)-k-3)$. This implies that $k \equiv 1 \pmod{3}$.
	
	On the other hand, suppose that all vertices of $G$ have order six, with the exception of the two vertices with order two. Then $6|(M(3,k)-1)$, which implies that $k$ is even.  
	
	Thus if $k \equiv 3 \pmod{6}$ or $k \equiv 5 \pmod{6}$, then no such digraph can exist.	   
\end{proof}

\begin{corollary}\label{no 3,k;+1 digraphs}
	If $k \geq 2$ is such that 
	\begin{itemize}
		\item $k \equiv 3$ or $5 \pmod{6}$,
		\item $k+2$ does not divide $\frac{M(3,k)-k-1}{2}$, and
		\item $M(3,k) + 1$ is prime,
	\end{itemize}	
	then there is no $(3,k;+1)$-digraph.
\end{corollary}
\begin{proof}
	Assume that $G$ is a $(3,k;+1)$-digraph such that $k$ satisfies each of these conditions. By Lemmas \ref{index 3 or more} and \ref{2 and 6}, $G$ is neither Type A nor Type B and hence must be outlier-regular. As the order of $G$ is prime, it follows that its outlier function $o$ consists of a single cycle of length $M(3,k)+1$; thus $G$ is vertex-transitive. However, any vertex-transitive digraph with prime order is a circulant digraph, which is not $k$-geodetic for $k \geq 2$. It follows that there is no $(3,k;+1)$-digraph for such $k$. 		
\end{proof}
The first $k$ for which Corollary \ref{no 3,k;+1 digraphs} applies are $k = 3, 15$ and $63$. This provides an independent proof of the non-existence of $(3,3;+1)$-digraphs \cite{MilMirSil}, as well as ruling out the existence of $(3,k;+1)$-digraphs for some larger $k$. 

\begin{corollary}
	There are no $(3,3;+1)$, $(3,15;+1)$- or $(3,63;+1)$-digraphs.
\end{corollary}

\section{Spectral results}

Now that we have more information about the permutation structure of $o$, we can apply some more powerful spectral results developed in \cite{MilMirSil}. If $A$ is the adjacency matrix of a $(d,k;+1)$-digraph $G$ with order $n$, $J$ is the $n \times n$ all-one matrix and $P$ is the permutation matrix associated with the permutation $o$, then counting the paths of length $\leq k$ we have that 
\begin{equation}\label{Counting paths for excess one}
	I+A+A^2+\dots +A^k = J-P.
\end{equation} 
We will now exploit the connection in Equation \ref{Counting paths for excess one} between the permutation structure of the outlier function $o$ of a $(d,k;+1)$-digraph $G$ and the spectrum of $G$. We will use the following concise description of the permutation structure of the outlier function $o$ from \cite{MilMirSil}.
\begin{definition}
	For any $(d,k;+1)$-digraph $G$ and $1 \leq j \leq M(d,k)+1$, the number of cycles of length $j$ in the permutation $o$ will be denoted by $m_j$. The $(M(d,k)+1)$-tuple $(m_1,m_2,\dots,m_{M(d,k)+1})$ is the \emph{permutation vector} of $G$. For $1 \leq j \leq M(d,k)+1$ we define
	
	\begin{itemize}
		\item $m'(j)$ is the number of odd cycles in the permutation $o$ with length divisible by $j$,
		\item $m''(j)$ is the number of even cycles in the permutation $o$ with length divisible by $j$, and
		\item $m(j) = m'(j)+m''(j)$ is the total number of cycles in $o$ with length divisible by $j$.
	\end{itemize} 
\end{definition}
Note that as the outlier function is fixed-point-free we always have $m_1 = 0$. We will also need the following family of polynomials derived from the cyclotomic polynomials.

\begin{definition}
	For $n,k \geq 1$ the polynomial $F_{n,k}(x)$ is defined by \[ F_{n,k}(x) = \Phi _n(1+x+x^2+\dots +x^k),\]
	where $\Phi _n(x)$ is the $n$-th cyclotomic polynomial.
	
\end{definition} 
In \cite{MilMirSil} Miller et al. derive the following relation between the characteristic polynomial of $J-P$ and the permutation vector of $G$.

\begin{lemma}\cite{MilMirSil}\label{excess one spectrum lemma}
	The characteristic polynomial of $J-P$ is 
	\[ (x-M(d,k))(x+1)^{-1}\prod _{j \geq 2, j \mbox{ even }}(x^j-1)^{m_j}\prod _{j \geq 3, j \mbox{ odd }}(x^j+1)^{m_j}. \] 
\end{lemma}

\begin{theorem}\label{2-outlier regular}
	There are no $2$-outlier-regular $(d,k;+1)$-digraphs.
\end{theorem}
\begin{proof}
	Assume that $G$ is a $2$-outlier-regular $(d,k;+1)$-digraph with order $n = M(d,k)+1$, i.e. the outlier function $o$ of $G$ contains only transpositions. Thus $m_2 = \frac{n}{2}$ and $m_i = 0$ for $i \not = 2$. Therefore by Lemma \ref{excess one spectrum lemma} the characteristic polynomial of $J-P$ is \[ (x-M(d,k))(x+1)^{-1}(x^2-1)^{n/2} = (x-M(d,k))(x-1)^{\frac{n}{2}}(x+1)^{\frac{n}{2}-1}. \]	
	It follows from Equation \ref{Counting paths for excess one} that the spectrum of $G$ consists of
	\begin{itemize}
		\item one eigenvalue $d$,
		\item $\frac{n}{2}$ eigenvalues $\lambda _i$, $1 \leq i\leq \frac{n}{2}$, such that $1+\lambda_i+\lambda_i^2+\dots +\lambda _i^k = 1$ for $1 \leq i \leq \frac{n}{2}$, and
		\item $\frac{n}{2}-1$ eigenvalues $\mu _i$ such that for $1 \leq i \leq \frac{n}{2}-1$ we have $1+\mu_i+\mu_i^2+\dots + \mu_i^k = -1$.
	\end{itemize} 
	For any integer $r \geq 0$ we define \[ \Lambda _r = \sum_{i=1}^{\frac{n}{2}}\lambda _i^r\]
	and  \[ M_r = \sum_{i=1}^{\frac{n}{2}-1}\mu _i^r.\]
	As for all vertices $u$ of $G$ we have $o^-(u) = o(u)$, the reasoning of Theorem \ref{divisibility for VT} shows that each vertex of $G$ is contained in $d$ directed $(k+1)$-cycles. By $k$-geodecity, any closed $(k+1)$-walk must be a cycle, so it follows by counting walks of length $\leq k+1$ that
	\begin{equation}\label{Trace r}
		\trace(A^r) = d^r +\Lambda _r +M_r = 0, 1 \leq r \leq k,
	\end{equation}
	and
	\begin{equation}\label{Trace k+1}
		\trace(A^{k+1}) = d^{k+1} +\Lambda _{k+1} +M_{k+1} = dn.
	\end{equation}
	
	Each eigenvalue $\lambda _i$ satisfies $1+\lambda_i + \lambda _i^2+\dots +\lambda _i^k = 1$; summing this geometric series and rearranging we obtain $\lambda _i(\lambda_i^k-1) = 0$, so each $\lambda _i$ is either zero or a $k$-th root of unity. Hence for $1 \leq r \leq k$ we have $\Lambda _{k+r} = \Lambda _r$. Similarly each eigenvalue $\mu _i$ satisfies $\mu_i ^{k+1}=2-\mu_i$, so that for $1 \leq r \leq k$ we have $\mu_i ^{k+r}=2\mu_i^{r-1}-\mu_i^r$. Thus for $1 \leq r \leq k$ the numbers $M_{k+r}$ satisfy $M_{k+r} = -M_r+2M_{r-1}$. 
	
	In particular $M_{k+1} = -M_1+2M_0 = -M_1+n-2$ and $\Lambda _{k+1} = \Lambda _1$. Therefore by Equation \ref{Trace k+1} we have
	\[ \Lambda _1 - M_1 = dn-d^{k+1}-n+2=d.\]
	Subtracting this from Equation \ref{Trace r} with $r = 1$ yields
	\[ M_1=-d.\]
	For any prime $p$ the polynomial $p+x+x^2+\dots +x^k$ is irreducible over $\mathbb{Q}$ \cite{Gim2}. As each $\mu _i$ is a solution of $2+x+x^2+\dots+x^k = 0$, it follows that the roots of $2+x+x^2+\dots+x^k$ must appear with equal multiplicity among the $\mu _i$. Therefore $k$ must divide $\frac{n}{2}-1$. The sum of the roots of $2+x+x^2+\dots+x^k$ is $-1$; therefore it follows that
	\[ -d = M_1 = -\frac{1}{k}\left[ \frac{n}{2}-1\right].\] 
	Rearranging, we obtain
	\[ n = 2+d+d^2+d^3+\dots +d^k = 2kd+2, \]
	or, simplifying,
	\[ 2k = 1+d+d^2+\dots +d^{k-1},\]
	which is impossible for $d,k \geq 2$.
\end{proof}
In particular, it follows from Lemma \ref{arc-transitive lemma} and Theorem \ref{2-outlier regular} that no $(d,k;+1)$-digraph with $d,k \geq 2$ can be arc-transitive.

\section{$2$-geodetic digraphs with excess one}\label{(d,2;+1)-digraphs}

In \cite{MilMirSil} the authors use spectral techniques to show that there are no 2-geodetic digraphs with excess one and degree $d \geq 8$. Sillasen's first paper on the subject \cite{Sil} proves that there are no $(2,2;+1)$-digraphs. Theorem \ref{no (3,2;+1)-digraphs} of the present work further showed that there are no $(3,2;+1)$-digraphs. This leaves open the existence of $(d,2;+1)$-digraphs for $d = 4,5,6$ and $7$. We will see that no $(d,2;+1)$-digraphs exist for these values of $d$. We first rule out the existence of outlier-regular $(d,2;+1)$-digraphs, then use an inductive approach to deal with the remaining cases.

Accordingly we shall now assume that any $(d,2;+1)$-digraph is outlier-regular. Then the index $\omega (G)$ of the digraph $G$ must be a non-unit divisor of the order $2+d+d^2$ of the $(d,2;+1)$-digraph. These divisors are displayed in Table \ref{tab:k = 2 excess 1}.

\begin{table}[h]\centering
	\begin{tabular}{|l|l|l|}
		\hline
		$d$ & $Order$ & Divisors $> 1$ \\
		\hline
		4 & 22 & \textcolor{red}{2},\textcolor{green}{11},\textcolor{blue}{22} \\
		5 & 32 & \textcolor{red}{2},\textcolor{orange}{4},\textcolor{orange}{8},\textcolor{orange}{16},\textcolor{blue}{32}\\
		6 & 44 & \textcolor{red}{2},4,\textcolor{green}{11},\textcolor{pink}{22},\textcolor{pink}{44} \\
		7 & 58 & \textcolor{red}{2},\textcolor{green}{29},\textcolor{blue}{58}\\
		\hline 
		
		\hline
	\end{tabular}
	\caption{Nontrivial divisors of the orders of the $(d,2;+1)$-graphs}
	\label{tab:k = 2 excess 1}
\end{table}

Theorem \ref{2-outlier regular} shows that there are no 2-outlier-regular digraphs, so we have already dealt with the divisors in red. Furthermore, if the index $\omega (G)$ of the digraph is equal to the order $2+d+d^2$ of the digraph $G$, then $G$ is vertex-transitive and by Corollary \ref{divisibility cor} the size $2d+d^2+d^3$ of $G$ must be divisible by $3$; this precludes the existence of outlier-regular digraphs with the divisors written in blue in Table \ref{2-outlier regular}. For the remaining possible structures of the outlier function we will need the exact factorisation of the characteristic polynomial of a $(d,2;+1)$-digraph from \cite{MilMirSil}.

\begin{lemma}\cite{MilMirSil}\label{k = 2 factorisation}
	The characteristic polynomial of a $(d,2;+1)$-digraph factorises in $\mathbb{Q} [x]$ as 
\[	\begin{aligned} (x-d)x^{a_1}(x+1)^{a_2}(x^2+x+2)^{\frac{m(2)+m'(1)-1}{2}}(x^2+1)^{m(4)} \\ \times \prod_{j \geq 3, j \mbox{ odd }}F_{j,2}(x)^{\frac{m''(j)}{2}}F_{2j,2}(x)^{\frac{m'(j)}{2}}\prod_{j \geq 6, j \mbox{ even }}F_{j,2}(x)^{\frac{m(j)}{2}},\end{aligned} \]
	where $a_1$ and $a_2$ are non-negative integers that satisfy the simultaneous equations 
	\begin{equation}\label{Sil simultaneous}
		a_1+a_2 = m''(1) \mbox{ and } d^2-d+1=a_1-a_2+2m(4).
	\end{equation}
\end{lemma}

\begin{lemma}\label{k = 2 omega odd}
	An outlier-regular $(d,2;+1)$-digraph cannot have odd index $\omega (G)$.
\end{lemma}
\begin{proof}
	Suppose that $\omega (G)$ is odd. Then $m(4) = m''(1) = 0$ and $a_1 = a_2 = 0$. Equation \ref{Sil simultaneous} then gives $d^2-d+1 = 0$, which has no real solutions. 
\end{proof}
Lemma \ref{k = 2 omega odd} disposes of all of the green entries in Table \ref{tab:k = 2 excess 1}.

\begin{lemma}\label{22 and 44}
	There are no $22$- or $44$-outlier-regular $(6,2;+1)$-digraphs.
\end{lemma}
\begin{proof}
	If a $(6,2;+1)$-digraph $G$ is $22$-outlier-regular, then $m(4) = 0$ and $m''(1) = 2$, so the simultaneous equations in Equation \ref{Sil simultaneous} give $a_1+a_2 = 2$ and $a_1-a_2 = 31$, which has no solution in non-negative integers. 
	
	Similarly, if $G$ is a $44$-outlier-regular $(6,2;+1)$-digraph, then $m''(1) = 1$ and $m(4) = 1$, so that Equation \ref{Sil simultaneous} yields $a_1+a_2 = 1$ and $a_1-a_2 = 29$, which again does not have non-negative solutions.  
\end{proof}

Lemma \ref{22 and 44} disposes of the pink divisors in Table \ref{tab:k = 2 excess 1}.

\begin{lemma}\label{no outlier regular 2-geo digraph}
	There are no outlier-regular $(d,2;+1)$-digraphs.
\end{lemma}
\begin{proof}
	First let $d = 5$ and $\omega \in \{ 4,8,16\} $. Then $m''(1) = m(4) = \frac{32}{\omega }$ and $d^2-d+1 = 21$. Equation \ref{Sil simultaneous} yields $a_1+a_2 = \frac{32}{\omega }$ and $a_1-a_2 = 21-\frac{64}{\omega }$. Solving for $a_1$ shows that \[ a_1 = \frac{1}{2}\left[ 21-\frac{32}{\omega }\right] ,\] 
	which is not an integer for $\omega \in \{ 4,8,16\} $. This gets rid of all of the orange entries in Table \ref{tab:k = 2 excess 1}.  
	
	The only remaining option for an $\omega $-outlier-regular $(d,2;+1)$-digraph $G$ is that $d = 6$ and $\omega = 4$. Then $m(4) = m''(1) = 44/4 = 11$ and $d^2-d+1 = 31$. Equation \ref{Sil simultaneous} becomes $a_1+a_2 = 11$ and $a_1-a_2 = 9$, which has solution $a_1 = 10$ and $a_2 = 1$. It follows from Lemma \ref{k = 2 factorisation} that the spectrum of $G$ is \[ \{ 6^{(1)},0^{(10)},-1^{(1)},\left[ \frac{-1+\sqrt{7}i}{2}\right ]^{(5)},\left[ \frac{-1-\sqrt{7}i}{2}\right ]^{(5)},i^{(11)},(-i)^{(11)}\} \]
	where multiplicities are indicated in round brackets. Summing the third powers of the eigenvalues, it follows that \[ Tr(A^3) = 240.\] A vertex is contained in $6$ directed triangles if it is Type I and $5$ directed triangles if it is Type II. If there are $\alpha $ Type $I$ vertices and $\beta $ Type II vertices in $G$, it follows that $6\alpha + 5\beta = 240$ and $\alpha + \beta = 44$. Solving these equations, we have $\alpha = 20$ and $\beta =24$. Let $A$ be the subdigraph of $G$ induced by the Type I vertices and $B$ the subdigraph induced by the Type II vertices. By Lemma \ref{Type-II outneighbours} it follows that $B$ consists of a collection of $6$ directed $4$-cycles. Thus $|(A,B)| = |(B,A)| = 120$. Thus each vertex in $A$ has out-neighbourhood entirely contained in $B$. Each vertex in $B$ has just one out-neighbour in $B$ and so each arc from a vertex $u$ of $A$ allows it to reach just twelve vertices of $B$ by paths of length $\leq 2$, which is impossible.     \end{proof}

It follows by Lemma \ref{omega is small} that any minimal $(d,2;+1)$-digraph must be either Type A (with a directed $4$-cycle $C$ of vertices with order $4$ and all other vertices with order greater than $4$) or Type B. We will take an inductive approach. Theorem \ref{no (3,2;+1)-digraphs} shows that there is no $(3,2;+1)$-digraph, so we can take any $(4,2;+1)$-digraph to be minimal, which allows us to show that $(4,2;+1)$-digraphs do not exist, so that any $(5,2;+1)$-digraph is minimal and so on. We make the following two observations from Lemma \ref{k = 2 factorisation} and Corollary \ref{automorphism fixset} respectively.

\begin{lemma}\label{observation one}
	If $G$ is a minimal $(d,2;+1)$-digraph, then $m''(1)$ is odd.
\end{lemma}
\begin{proof}
	By Lemma \ref{k = 2 factorisation} we have $a_1+a_2 = m''(1)$ and $a_1-a_2 = d^2-d+1-2m(4)$, so $m''(1)$ has the same parity as $d^2-d+1$, which is odd.
\end{proof}

\begin{lemma}\label{observation two}
	There are exactly two non-zero entries in the permutation vector of a minimal $(d,2;+1)$-digraph and both cycle lengths of $o$ are even.
\end{lemma}
\begin{proof}
	By Lemma \ref{no outlier regular 2-geo digraph}, a minimal $(d,2;+1)$-digraph is either Type A, in which case the smallest non-zero entry of the permutation vector is $m_4 = 1$, or Type B, in which case the smallest entry is $m_2 = 1$. 
	
	By Corollary \ref{automorphism fixset}, for each $r \geq 1$ the automorphism $o^r$ has fix-set of size $0$, $2$ or $4$, or else fixes every vertex of $G$. Therefore if for some $r \geq 2$ the automorphism $o^r$ fixes $3$ or $\geq 5$ vertices of $G$, then $o^r$ is the identity automorphism and every vertex of $G$ has order dividing $r$. Suppose that the permutation vector contains a non-zero entry $m_j = 0$, where $j\geq 3$ is odd. Then $o^j$ fixes either $3$ or $\geq 5$ vertices of $G$, but not the vertices with even order, which is impossible. Likewise, if $i < j$ are both even, $i$ is greater than $\omega (G)$ and $m_i$ and $m_j$ are both non-zero, then $o^i$ fixes at least $6$ vertices with order $\omega (G)$ or $i$, but not the vertices with order $j$, again a contradiction.  
\end{proof}

\begin{theorem}\label{no 4,2;1 digraph}
	There are no $(4,2;+1)$-digraphs.
\end{theorem}
\begin{proof}
	Suppose that $G$ is a $(4,2;+1)$-digraph. Assume first that $G$ is Type A. Let $u$ be a vertex on the 4-cycle $C$ of vertices with order $4$, with $N^+(u) = \{ u_1,u_2,u_3,u_4\} $, where $u_1$ also lies on $C$. The automorphism $o^4$ fixes $u$ and $u_1$, but has no fixed points in $\{ u_2,u_3,u_4\} $, so $o^2$ permutes $u_2,u_3$ and $u_4$ in a $3$-cycle, say $(u_2u_3u_4)$. Thus $o^{12}$ fixes every vertex of $T(u)$; by Corollary \ref{automorphism fixset}, $o^{12}$ fixes every vertex of $G$ and so every vertex of $G$ has order $4, 6$ or $12$. $G$ has order $22$, so we either have $m_4 = 1, m_6 = 3$ or $m_4 = 1, m_6 = 1, m_{12} = 1$. Both are impossible by Lemmas \ref{observation one} and \ref{observation two}.
	
	Thus we can assume $G$ to be Type B. Let $u$ be one of the two vertices of $G$ with order $2$. The automorphism $o^2$ fixes $u$, but permutes its out-neighbours $u_1,u_2,u_3$ and $u_4$ without fixed points. Hence without loss of generality $o^2$ permutes these vertices either as $(u_1u_2)(u_3u_4)$ or $(u_1u_2u_3u_4)$; in either case $o^8$ fixes every vertex of $T(u)$ and hence all of $G$, so every vertex has order $2,4$ or $8$. Hence $22  = 2+4m_4+8m_8$, or $5 = m_4+2m_8$. There are three solutions of this equation: i) $m_4 = 1, m_8 = 2$, ii) $m_4 = 3, m_8 = 1$ and iii) $m_4 = 5, m_8 = 0$. By Lemma \ref{observation two} only option iii) can hold. Thus the two non-zero entries of the permutation vector of $G$ are $m_2 = 1, m_4 = 5$ and $m''(1) = 6$, which is even, contradicting Lemma \ref{observation one}.
\end{proof}

Having proved that there are no $(4,2;+1)$-digraphs, we know that any $(5,2;+1)$-digraph is minimal.
\begin{theorem}
	There is no $(5,2;+1)$-digraph.
\end{theorem}
\begin{proof}
	Assume that $G$ is a $(5,2;+1)$-digraph. $G$ has order $32$. By Theorem \ref{no 4,2;1 digraph}, $G$ is minimal and hence by Lemma \ref{no outlier regular 2-geo digraph} is either Type A or Type B. Suppose that $G$ is Type A. As in Theorem \ref{no 4,2;1 digraph}, fix a vertex $u$ on the cycle $C$ of vertices with order $4$ and set $N^+(u) = \{ u_1,u_2,u_3,u_4,u_5\} $, where $u_1 \in V(C)$. The automorphism $o^4$ must permute $u_2,u_3,u_4$ and $u_5$ amongst themselves without fixed points, so $o^4$ acts on these vertices either as $(u_2u_3u_4u_5)$ or $(u_2u_3)(u_4u_5)$; in either case $o^{16}$ fixes all vertices of $G$ and every vertex order is $4,8$ or $16$. We have $32 = 4+8m_8+16m_{16}$, or $7 = 2m_8+4m_{16}$. However, the right-hand side is even and the left odd.
	
	Now suppose that $G$ is Type $B$ and let $u$ be a vertex of $G$ belonging to the unique transposition of $o$. $o^2$ permutes the vertices of $N^+(u) = \{ u_1,u_2,u_3,u_4,u_5\} $ without fixed points; without loss of generality, $o^2$ acts on these vertices either as i) $(u_1u_2u_3u_4u_5)$ or ii) $(u_1u_2u_3)(u_4u_5)$. 
	
	In case i) $o^{10}$ fixes every vertex of $G$ and every vertex order is $2,5$ or $10$, where $30 = 5m_5+10m_{10}$. By Lemma \ref{observation two}, $m_5 = 0$, so the non-zero entries of the permutation vector are $m_2 = 1, m_{10} = 3$, yielding $m''(1) = 4$, contradicting Lemma \ref{observation one}.
	
	In case ii) $o^{12}$ is the identity and every vertex has order $2,3,4,6$ or $12$. Lemma \ref{observation two} shows that $m_3 = 0$. We have $30 = 4m_4+6m_6+12m_{12}$ and Lemma \ref{observation two} shows that just one of $m_4,m_6$ and $m_{12}$ is non-zero. By Lemma \ref{observation two}, as $12$ and $4$ do not divide $30$, we have $m_4 = m_{12} = 0$. If $m_6 > 0$, then $m_2 = 1$ and  $m_6 = 5$, giving an even value of $m''(1)$, which is impossible. 
\end{proof}

\begin{theorem}
	There is no $(6,2;+1)$-digraph.	
\end{theorem}
\begin{proof}
	Suppose that there exists a $(6,2;+1)$-digraph $G$ with order $44$. Suppose that $G$ is Type $A$; as before, let $u \rightarrow u_1$ be an arc of the $4$-cycle of vertices with order $4$. We can assume that the automorphism $o^4$ permutes the other out-neighbours $\{ u_2,u_3,u_4,u_5,u_6\} $ either as i) $(u_2u_3u_4u_5u_6)$ or ii) $(u_2u_3u_4)(u_5u_6)$. In case i) every vertex of $G$ has order $4, 5, 10$ or $20$; by Lemma \ref{observation two} $m_5 = 0$. If $m_{20} > 0$, then by Lemma \ref{observation two} $m_4 = 1, m_{10} = 0, m_{20} = 2$. By Lemma \ref{k = 2 factorisation}, this yields $a_1+a_2 = 3$ and $a_1-a_2=25$, which would imply that $a_2$ is negative. Therefore $m_{10} = 4, m_{20} = 0$, giving $a_1+a_2 = 5, a_1-a_2 = 29$, which again is impossible.

	In case ii) every vertex order is $4,6,8,12$ or $24$. We have $40 = 6m_6+8m_8+12m_{12}+24m_{24}$. As none of $6,12$ or $24$ are divisors of $40$, Lemma \ref{observation two} shows that $m_8 = 5,m_6 = m_{12} = m_{24} = 0$, so that $m''(1) =6$ is even, violating Lemma \ref{observation one}. 
	
	Now suppose that $G$ is Type B. Let $u$ be a vertex in the unique transposition of $o$. We can assume that $o^2$ permutes the vertices of $N^+(u)$ in one of four ways: i) $(u_1u_2u_3u_4)(u_5u_6)$, ii) $(u_1u_2u_3u_4u_5u_6)$, iii) $(u_1u_2u_3)(u_4u_5u_6)$ or iv) $(u_1u_2)(u_3u_4)(u_5u_6)$. For Case i), every vertex has order $2$, $4$ or $8$, but neither $4$ nor $8$ divide $42$, contradicting Lemma \ref{observation two}. 
	
	In Cases ii), iii) and iv) each vertex order is $2,3,4,6$ or $12$, with $42 = 3m_3+4m_4+6m_6+12m_{12}$. By Lemma \ref{observation two} we know that $m_3 = 0$ and just one of $m_4,m_6$ and $m_{12}$ is non-zero. As $4$ and $12$ do not divide $42$, we must have $m_4 = m_{12} = 0$ and $m_6 = 7$, so that $m''(1) = 8$ is even, a contradiction.  
\end{proof}

\begin{theorem}
	There are no $(7,2;1)$-digraphs.
\end{theorem}
\begin{proof}
	Assume that $G$ is a $(7,2;+1)$-digraph. $G$ has order $58$. Suppose that $G$ is Type $A$ and $u \rightarrow u_1$ is an arc of the $4$-cycle of vertices with order $4$. Then we can assume that $o^4$ permutes $\{ u_2,u_3,u_4,u_5,u_6,u_7\} = N^+(u)-\{ u_1\} $ in one of the following ways: i) $(u_2u_3u_4u_5)(u_6u_7)$, ii) $(u_2u_3)(u_4u_5)(u_6u_7)$, iii) $(u_2u_3u_4u_5u_6u_7)$ or iv) $(u_2u_3u_4)(u_5u_6u_7)$.  
	
	In Cases i) and ii) every vertex order is $4, 8$ or $16$ and $54 = 8m_8+16m_{16}$. However neither $8$ nor $16$ divides $54$, violating Lemma \ref{observation two}. We can thus assume that either case iii) or iv) holds and every vertex order is $4,6,8,12$ or $24$, with $54 = 6m_6+8m_8+12m_{12}+24m_{24}$. Lemma \ref{observation two} shows that $m_8 = m_{12} = m_{24} = 0$ and $m_6 = 9$. Then $m''(1) = 10$ is even, contradicting Lemma \ref{observation one}.
	
	Therefore assume that $G$ is Type $B$ and let $u$ be a vertex with order $2$. $o^2$ permutes the elements of $N^+(u) = \{ u_1,u_2,u_3,u_4,u_5,u_6,u_7\} $ in one of the following ways: i) $(u_1u_2u_3u_4u_5u_6u_7)$, ii) $(u_1u_2u_3u_4u_5)(u_6u_7)$, iii) $(u_1u_2u_3u_4)(u_5u_6u_7)$ or iv) $(u_1u_2u_3)(u_4u_5)(u_6u_7)$. 
	
	In Case i) all vertex orders are $2,7$ or $14$, so by Lemma \ref{observation two}, $m_2 = 1, m_7 = 0, m_{14} = 4$. Then $m''(1) = 5$ and $m(4) = 0$, so that $a_1+a_2 = 5$ and $a_1-a_2 = 43$, which has no suitable solutions. 
	
	In Case ii) all vertex orders are $2,4,5,10$ or $20$, $m_5 = 0$ and $56 = 4m_4+g10m_{10}+20m_{20}$. $10$ and $20$ do not divide $56$, so $m_2 = 1, m_4 = 14$. In Cases iii) and iv) all vertex orders are $2,3,4,6,8,12$ or $24$ and $56 = 3m_3+4m_4+6m_6+8m_8+12m_{12}+24m_{24}$. Lemma \ref{observation two} shows that the only valid solutions are $m_2 = 1$ and $m_8 = 7$ and again $m_2 = 1, m_4 = 14$. In the former case $m''(1)$ is even, so we have shown that we can assume that $m_2 = 1$ and $m_4 = 14$. Thus $m''(1) = 15$ and $m(4) = 14$, giving $a_1+a_2 = 15$, $a_1-a_2 = 43-28 = 15$, giving $a_1 = 15$ and $a_2 = 0$.  It follows from Lemma \ref{k = 2 factorisation} that the spectrum of $G$ is, counted by multiplicity,
	\[ \{ 7,0^{(15)},\left[ \frac{-1+\sqrt{7}i}{2} \right]^{(7)},\left[ \frac{-1-\sqrt{7}i}{2} \right]^{(7)},i^{(14)},-i^{(14)}\} .\] 
	Adding up the third powers of these eigenvalues, we see that the trace of $A^3$ is $Tr(A^3) = 378$. A vertex lies in $7$ directed triangles if it is Type I and 6 if it is Type II; therefore if there are $\alpha $ Type I vertices in $G$ and $\beta $ Type II vertices, then counting paths we see that 
	\[ 7\alpha +6\beta = 378, \alpha +\beta = 58.\]
	Solving these equations gives $\alpha = 30, \beta = 28$. As a vertex has the same type as its outlier, it follows that the subdigraph $B$ induced by the Type II vertices consists of 7 directed 4-cycles and there is a set $A$ of 28 vertices $v$ such that $d(v,o^-(v)) = 2$, the other two Type I vertices being $u$ and $o(u)$. 
	
	$B$ has size 28 and $(B,V(G)-B) = (V(G)-B,B) = 168$. This means that the subdigraph induced by $A' = A \cup \{ u,o(u)\} $ has size 42. Fix a vertex $v \in A'$. The outlier $o(v)$ of $v$ also lies in $A'$, so $v$ can reach every vertex of $B$ by $\leq 2$-paths. The vertex $v$ has 7 out-going arcs. Suppose that the out-degree of $v$ in the subdigraph induced by $A'$ is $\leq 2$. Each arc from $v$ to $B$ allows $v$ to reach 2 vertices of $B$; therefore the largest possible number of vertices of $B$ that $v$ could reach by $\leq 2$-paths would be achieved if $v$ has two out-neighbours in $A'$, each of which has out-neighbourhood contained in $B$ ; however, this would still only allow $v$ to reach 24 of the 28 vertices of $B$. It follows that the minimum out-degree in $A'$ is $\geq 3$, which implies that the size of the subdigraph induced by $A'$ is $\geq 90$, a contradiction.  
\end{proof}
This completes the remaining cases in the classification of $(d,2;+1)$-digraphs from \cite{MilMirSil}. Combined with the results of \cite{MilMirSil}, we see that to find a digraph with excess one, we must look at digraphs that have degree at least three and are at least 5-geodetic.

\begin{theorem}
	If $d,k \geq 2$ and $\epsilon (d,k) = 1$, then $d \geq 3$ and $k \geq 5$.
\end{theorem}

	\section*{Acknowledgements}
The author acknowledges funding from an LMS Early Career Fellowship (Project ECF-2021-27) and thanks Prof. \v Sir\a'a\v n and Dr. Erskine for helpful discussion of this research.



\begin{thebibliography}{9}
	
	\bibitem{BanIto1} Bannai, E. and Ito, T., \emph{On finite Moore graphs.} J. Fac. Sci. Tokyo Univ, 20 (191-208) (1973), 80.
	
	\bibitem{BanIto} Bannai, E. and Ito, T., \emph{Regular graphs with excess one.} Discrete Math. 37 (2-3) (1981), 147-158.
	
	\bibitem{BasChoMil} Baskoro, E. T., Cholily, Y. M. and Miller, M., \emph{Enumerations of vertex orders of almost Moore digraphs with selfrepeats.} Discrete Mathematics, 308 (1) (2008), 123-128.
	
	\bibitem{BasMilPleZna} Baskoro, E.T., Miller, M., Plesn\'{i}k, J. and Zn\'{a}m, \S., \emph{Digraphs of degree 3 and order close to the Moore bound.} J. Graph Theory 20 (3) (1995), 339-349.
	
	\bibitem{BasMilSirSut} Baskoro, E.T., Miller, M., \v Sir\a'a\v n, J. and Sutton, M., \emph{Complete characterization of almost Moore digraphs of degree three.} J. Graph Theory 48 (2005), 112-126.
	
	
	\bibitem{BonMur} Bondy, J.A. and Murty, U.S.R., \emph{Graph theory with applications.} (Vol. 290) (1976), London: Macmillan.

	\bibitem{BriTou} Bridges, W.G. and Toueg, S., \emph{On the impossibility of directed Moore graphs.} J. Comb. Theory B29 (1980), 339-341.
	
	\bibitem{ChoBasUtt} Cholily, Y.M., Baskoro, E.T. and Uttunggadewa, S., \emph{Some conditions for the existence of $(d,k)$-digraphs.} In Indonesia-Japan Joint Conference on Combinatorial Geometry and Graph Theory, Springer, Berlin, Heidelberg (2003), 87-93.
	
	
	\bibitem{ConGimGonMirMor} Conde, J., Gimbert, J., Gonz\`alez, J., Miret, J.M. and Moreno, R., \emph{Nonexistence of almost Moore digraphs of diameter four.} Electron. J. Comb. 20 (1) (2013).
	
	\bibitem{ConGimGonMirMor2} Conde, J., Gimbert, J., Gonz\`alez, J., Miret, J.M. and Moreno, R., \emph{Nonexistence of almost Moore digraphs of diameter three.} Electron. J. Comb. 15 (2008).
	
	\bibitem{Dam} Damerell, R.M., \emph{On Moore graphs.} Math. Proc. Cambridge Philos. Soc. 74 (2), Cambridge University Press (1973), 227-236.
	
	\bibitem{ExoJaj} Exoo, G. and Jajcay, R., \emph{Dynamic cage survey.} Electron. J. Comb. 1000 (2011), DS16-May.
	
	\bibitem{FioYebAle} Fiol, M.A., Alegre, I. and Yebra, J.L.A., \emph{Line digraph iterations and the $(d,k)$-digraph problem for directed graphs.} Proc. 10th Int. Symp. Comput. Architecture (1983), 174-177.
		
	\bibitem{Gim} Gimbert, J., \emph{Enumeration of almost Moore digraphs of diameter two.} Discrete Math. (231) (2001), 177-190.
	
	\bibitem{Gim2} Gimbert, J., \emph{On the existence of $(d,k)$-digraphs.} Discrete Math. 197 (1999), 375-391.
		
	\bibitem{HofSin} Hoffman, A.J. and Singleton, R.R., \emph{On Moore graphs with diameter 2 and 3.} IBM J. Res. Develop. 4 (1960), 497-504.
	
	\bibitem{MilFri} Miller, M. and Fri\v s, I., \emph{Maximum order digraphs for diameter 2 or degree 2.} Pullman Volume of Graphs and Matrices, Lecture Notes in Pure and Applied Math. 139 (1992), 269-298.
		
	\bibitem{MilGimSirSla} Miller, M., Gimbert, J., \v Sir\a'a\v n, J. and Slamin, \emph{Almost Moore digraphs are diregular.} Discrete Math. 218, No. 1-3 (2000), 265-270.
	
	\bibitem{MilMirSil} Miller, M., Miret, J.M. and Sillasen, A.A., \emph{On digraphs of excess one.} Discrete Appl. Math. (238) (2018), 161-166.
	
	\bibitem{MilSir} Miller, M. and \v Sir\a'a\v n, J., \emph{Moore graphs and beyond: A survey of the degree/diameter problem.} Electron. J. Comb., Dynamic Survey DS14 (2005).
	
	\bibitem{NguMilGim} Nguyen, M.H., Miller, M. and Gimbert, J., \emph{On mixed Moore graphs.} Discrete Math. 307 (7) (2007), 964-970.
	
	\bibitem{SilThesis} Sillasen, A.A., \emph{Digraphs of small defect or excess.} Department of Mathematical Sciences, Aalborg	University. (Ph.D. Report Series; No. 26 - 2014).
	
	\bibitem{Sil} Sillasen, A.A., \emph{On $k$-geodetic digraphs with excess one.} Electron. J. Graph Theory Appl. 2 (2) (2014), 150-159.
	
	\bibitem{Sil2} Sillasen, A.A., \emph{Subdigraphs of almost Moore digraphs induced by the fixpoints of an automorphism.} Electron. J. Graph Theory Appl. 3 (1) (2015), 1-7.
		
	\bibitem{Tui} Tuite, J., \emph{Digraphs with degree two and excess two are diregular.} Discrete Math. 342 (5) (2019), 1233-1244.
	
	\bibitem{Tui2} Tuite, J., \emph{On diregular digraphs with degree two and excess three.} Discrete Appl. Math. 266 (2019), 331-339.
	
	\bibitem{Tui3} Tuite, J., \emph{On diregular digraphs with degree two and excess two.} Discrete Appl. Math. 238 (2018), 167-174.	
	
	\bibitem{TuiErs} Tuite, J. and Erskine, G., \emph{On networks with order close to the Moore bound.} Preprint (2021)
	
\end{thebibliography}
\end{document}